\documentclass[a4paper,12pt]{scrartcl}
\usepackage{amsmath,amssymb,latexsym,amsthm,enumerate}
\usepackage{xcolor}
\usepackage{graphicx}
\usepackage{epstopdf}
\usepackage[colorlinks,linkcolor=red,citecolor=green]{hyperref}

\usepackage[T1]{fontenc}
\usepackage[utf8]{inputenc}
\usepackage[english]{babel}

\newcommand*{\wh}{\widehat}

\newcommand*{\ol}{\overline}

\newcommand*{\N}{\mathbb{N}}
\newcommand*{\R}{\mathbb{R}}
\newcommand*{\C}{\mathbb{C}}

\newcommand*{\Sd}{\mathbb{S}^d}

\newcommand*{\cF}{\mathcal{F}}

\DeclareMathOperator*{\esssup}{ess\,sup}

\newcommand*{\diag}{\operatorname{diag}}

\newcommand*{\Id}{\operatorname{Id}}

\renewcommand*{\Re}{\operatorname{Re}}
\renewcommand*{\Im}{\operatorname{Im}}

\newcommand{\be}{\begin{eqnarray*}}
\newcommand{\ee}{\end{eqnarray*}}
\newcommand{\ben}{\begin{eqnarray}}
\newcommand{\een}{\end{eqnarray}}
\newcommand{\bi}{\begin{itemize}}
\newcommand{\ei}{\end{itemize}}

\newtheorem{theo}{Theorem}[section]

\newtheorem{lemma}[theo]{Lemma}
\newtheorem{propo}[theo]{Proposition}
\newtheorem{corollary}[theo]{Corollary}

\theoremstyle{definition}

\newtheorem{defi}[theo]{Definition}

\newtheorem{remark}[theo]{Remark}
\newtheorem{assumption}{Assumption}

\newcounter{numpar}[section]

\title{Inhomogeneous affine Volterra processes}

\author{Julia Ackermann\thanks{Mathematisches Institut, Justus-Liebig-Universit\"{a}t Giessen  } \and 
Thomas Kruse \thanks{Mathematisches Institut, Justus-Liebig-Universit\"{a}t Giessen }
\and Ludger Overbeck\thanks{Mathematisches Institut, Justus-Liebig-Universit\"{a}t Giessen  }
}

\begin{document}

\maketitle
\begin{abstract}
We extend recent results on affine Volterra processes to the inhomogeneous case. This includes moment bounds of solutions of Volterra equations driven by a Brownian motion with an inhomogeneous kernel $K(t,s)$ and inhomogeneous drift and diffusion coefficients $b(s,X_s)$ and $\sigma(s,X_s)$. In the case of affine $b$ and $\sigma \sigma^T$ we show how the conditional Fourier-Laplace functional can be represented by a solution of an inhomogeneous Riccati-Volterra integral equation. 
For a kernel of convolution type $K(t,s)=\ol {K}(t-s)$ we establish existence of a solution to the stochastic inhomogeneous Volterra equation. 
If in addition $b$ and $\sigma \sigma^T$ are affine, we prove that the conditional Fourier-Laplace functional is exponential-affine in the past path. 
Finally, we apply these results to an inhomogeneous extension of the rough Heston model used in mathematical finance.
\end{abstract}

\section*{Introduction}
The purpose of the paper is fourfold. First, we study in Section \ref{sec:mom_bounds} general $d$-dimensional stochastic inhomogeneous Volterra equations 
\begin{equation}\label{eq:volterra_d_volt_inhom_intro}
X_t=X_0+\int_0^tK(t,s)b(s,X_s)ds+\int_0^tK(t,s)\sigma(s,X_s)dW_s, \quad t \in [0,T],
\end{equation}
where $W$ 
is an $m$-dimensional Brownian motion. Second, in Section~\ref{sec:fully_inhom_volt_case_main_thm}, we analyse the conditional Fourier-Laplace functional  
\begin{equation}\label{eq:fl_func_intro}
	E\left[\exp\left(uX_T + \int_0^T f(s) X_s ds \right)\,\bigg|\, \cF_t \right], \quad t\in[0,T],\end{equation}
  in the affine case
\begin{equation}\label{eq:affine_intro}
b(s,x) = b^0(s) + \sum_{i=1}^d b^{i}(s)x_i, \quad
\sigma(s,x)\sigma(s,x)^T=a(s,x) = A^0(s) + \sum_{i=1}^d A^{i}(s) x_i.
\end{equation}
In the homogeneous affine case, where the kernel is of convolution type $K(t,s)=\overline K(t-s)$ and $b$ and $\sigma$ do not depend on time, Abi Jaber, Larsson \& Pulido \cite{Jaber2019affine} show that, remarkably, \eqref{eq:fl_func_intro}
admits a representation in terms of an associated Riccati-Volterra equation and thereby extend the exponential-affine formula for classical affine diffusions (see, e.g., Duffie, Filipovi{\'c} \& Schachermayer \cite{Duffie2003affineProcFinance}) to the Volterra case.
In our main result, Theorem \ref{thm:multdim_volt_kts}, we extend this result to the inhomogeneous case and show 
that the functional \eqref{eq:fl_func_intro} can be expressed in terms of the solution $\psi$ to a  Riccati-type equation
\begin{equation*}
	\begin{split}
	\psi(t)&=uK(T,t)+\int_t^{T} \left(f(s) + \psi(s)B(s)+\frac{1}{2}A(s,\psi(s))\right) K(s,t)ds, \quad t \in [0,T],
	\end{split}
	\end{equation*}  
where $B$ and $A$ are based on the coefficients $b^i$ and $A^i$ in \eqref{eq:affine_intro}.  
Third, in Section~\ref{sec:conv_kernel}, assuming a convolution kernel $K(t,s)=\overline K(t-s)$ we show existence of a solution to \eqref{eq:volterra_d_volt_inhom_intro} and under the affine assumption \eqref{eq:affine_intro} we prove that the conditional Fourier-Laplace functional is exponential-affine in the past path $(X_s)_{s\in [0,t]}$.
Finally, in Section~\ref{sec:inhom_volt_heston_model}, we apply our results to extend the rough Heston model introduced and studied by El Euch \& Rosenbaum in \cite{Euch2018perfectHedging} and \cite{Euch2019characteristicFctRoughHeston} to the inhomogeneous case, which is used in mathematical finance, cf., e.g., Alfeus, Overbeck \& Schlögl~\cite{Alfeus2019regimeSwitching} and Alfeus, Nikitopoulos \& Overbeck~\cite{Alfeus2021forwardVol}.
 
In the homogeneous case, comparable results can be found in 
Abi Jaber, Larsson \& Pulido \cite{Jaber2019affine}, which serves as our basic reference. However, the inhomogeneous character requires to modify the statement of  the results and  the strategies to prove them. The result on moment bounds of solutions $(X_t)_{t\in [0,T]}$ of the stochastic inhomogeneous Volterra equation, Lemma~\ref{lem:moments_of_soln_bdd}, extends the results in the homogeneous case provided that the corresponding assumptions in 
\cite{Jaber2019affine} on the coefficients $b$ and $\sigma$ hold uniformly in $t$ (see Assumptions \ref{assumption:kernel_L_2} and \ref{assumption:linear_growth} below). 
A key observation of our paper is that one can establish the representation result for the conditional 
Fourier-Laplace functional \eqref{eq:fl_func_intro}, Theorem~\ref{thm:multdim_volt_kts}, and the result that \eqref{eq:fl_func_intro} is exponential-affine in the past path, Theorem~\ref{thm:represent_for_Y}, without making use of the so-called resolvent of the second kind of the kernel. Indeed, in contrast to \cite{Jaber2019affine}, our proofs of Theorem~\ref{thm:multdim_volt_kts} and Theorem~\ref{thm:represent_for_Y} do neither rely on this resolvent nor make use of a variation of constants formula for $X$. This observation leads to streamlined versions of the proofs already in the homogeneous case. 
If $K$ is of convolution type, i.e., $K(t,s)=\ol K(t-s)$, we  prove  existence of a solution of  \eqref{eq:volterra_d_volt_inhom_intro} by applying existence results from \cite{Jaber2019affine} to the then homogeneous time-space process $(t, X_t)$. We require the existence of the resolvent of the first kind $L$ for the kernel only in order to show in Theorem \ref{thm:represent_for_Y} that the conditional Fourier-Laplace functional \eqref{eq:fl_func_intro} is exponential-affine in the past path.

In general, affine processes build an important class of stochastic processes because despite being non-Gaussian they still exhibit some analytic tractability. In particular, their characteristic functions can be solved by ordinary differential or integral equations, so-called Riccati equations. Special instances of affine processes are also used in applications of stochastic analysis such as mathematical finance and in models in biology since a long time. First important papers are Feller~\cite{Feller1951}, Kawazu \& Watanabe~\cite{KawazuWatanabe} and Cox, Ingersoll \& Ross~\cite{CIR}.  Milestones in the development of the theory of affine processes are Duffie, Filipovi{\'c} \& Schachermayer~\cite{Duffie2003affineProcFinance} and Filipovi{\'c} \& Mayerhofer~\cite{filipovic2009affine}. Recent papers on affine processes consider matrix valued processes in Cuchiero, Filipovi{\'c}, Mayerhofer \& Teichmann~\cite{Cuchiero2011affineProcPosSemidef} and representations of fractal processes by affine processes in Harms \& Stefanovits~\cite{Harms2019affineRepresFractional}.

On the other hand, also Volterra equations are useful in many mathematical and applied topics. One of the useful features is that the dynamic might depend not only on the running time $s$ but also on $t$ via the kernel $K$ in equation \eqref{eq:volterra_d_volt_inhom_intro}. However, then the Markov property is lost. First systematic  treatments of stochastic Volterra integral equations can be found in Berger \& Mizel~
\cite{Berger1980aVoltEqIto} and \cite{Berger1980bVoltEqIto}, which in Protter~\cite{Protter1985voltEq} are extended to a semimartingale setting. In Pardoux \& Protter~\cite{Pardoux1990stochVolt} anticipative coefficients are introduced, and Volterra equations in an infinite dimensional context are analysed in Zhang~ \cite{Zhang2010stochVolt} and on Wiener spaces in Decreusefond~\cite{Decreusefond2002regProp}.

As a third feature, besides the affine and Volterra structure,  we consider inhomogeneous coefficients. For Volterra-type equations, inhomogeneity is natural since the coefficients depend on $t$ and $s$, and therefore already the early papers are formulated in a general inhomogeneous setting, e.g., Protter~\cite{Protter1985voltEq}. But so far mainly equations of type \eqref{eq:volterra_d_volt_inhom_intro} were analysed where $b$ and $\sigma$ do not depend on $s$ or where $K$ is of convolution type, i.e., $K(t,s)=\ol K(t-s)$. 
For example, Abi Jaber, Cuchiero, Larsson \& Pulido~\cite{Jaber2019weak} establish weak existence results for stochastic convolution equations with jumps avoiding the use of resolvents of first or second kind and Abi Jaber~\cite{Jaber2020weakL1} provides existence and uniqueness results for such equations with $L^1$-kernels. 
In Abi Jaber~\cite{Jaber2020laplaceTrafoWishart} and Abi Jaber~\cite{Jaber2020charactFctAnalytic} kernels of type $K(t,s)$ in a Gaussian setting are considered, with linear drift $b(s,x)=bx$ and constant diffusion term. The basic equation in Coutin \& Decreusefond~\cite{Coutin2001stochVolt} is as  \eqref{eq:volterra_d_volt_inhom_intro} but regularity on the kernel $K$ and Lipschitz-conditions on $b$ and $\sigma$ are imposed, leaving out affine specifications. The starting point in the recent work Jacquier \& Pannier~\cite{Jacquier2020largeDev} is also  a general equation such as \eqref{eq:volterra_d_volt_inhom_intro}; however, existence of a solution is basically assumed. Also the equation handled in Wang~\cite{Wang2008existence} concerns an inhomogeneous equation, but the result is not directly applicable in our affine case.  
In the context of affine processes the results from Duffie, Filipovi{\'c} \& Schachermayer ~\cite{Duffie2003affineProcFinance} are extended to the inhomogeneous case in Filipovi{\'c}~\cite{filipovic2005affineinhomogeneous}.

One motivation to consider affine Volterra processes comes from volatility modeling in mathematical finance. There, the Volterra equation is mainly used with the specification $K(t,s)=\frac{(t-s)^{\alpha-1}}{\Gamma(\alpha)},\alpha\in (\frac{1}{2},1)$, giving rise to rough volatility processes. Recent research shows that these rough phenomena are important to give a more realistic description of the priced and observed volatility, cf., e.g., \cite{Jaber2019multifactorApprox, Jaber2019markovianStruct,Jaber2019liftingHeston, Euch2019characteristicFctRoughHeston, Euch2019rougheningHeston, Gatheral2018volatilityIsRough, Livieri2018roughVolaEvidence, Euch2018perfectHedging, Comte2012affineFractionalVola, Gatheral2019affineForwardVariance, Fukasawa2019isVolaRough, Jaisson2016roughHawkes, KellerRessel2018affineRoughModels, Alfeus2019regimeSwitching, Alfeus2021forwardVol}. 
Time-inhomogeneous parameters in volatility modelling are considered in El Euch \& Rosenbaum~\cite{Euch2018perfectHedging} where the conditional characteristic function of a rough Heston model with time-inhomogeneous mean-reversion level is studied, in Alfeus, Overbeck \& Schlögl~\cite{Alfeus2019regimeSwitching} where this is applied to a regime switching model and in Alfeus, Nikitopoulos \& Overbeck~ \cite{Alfeus2021forwardVol} where a rough Heston model with time-dependent volatility is required. The analytic requirements of those  models can be handled by our results in Section \ref{sec:inhom_volt_heston_model} on the inhomogeneous rough Heston model.

\bigskip

\textbf{Notation:} 
For every $m,n\in \N$ we denote by $\lVert \cdot\rVert$ the Frobenius norm $\lVert \cdot \rVert \colon \C^{m\times n} \to [0,\infty)$, $\lVert A\rVert^2= \sum_{k=1}^m \sum_{j=1}^n (\Re A_{kj})^2 + (\Im A_{kj})^2$ where $A=(A_{kj})_{k =1,\ldots,m; j=1,\ldots,n} \in \C^{m\times n}$. 
By convention, for $d \in \N$, vectors in $\C^d$ are column vectors, whereas elements of $(\C^d)^*$ are row vectors. 
For $m,n \in \N$ and $A = (A_{kj})_{k =1,\ldots,m; j=1,\ldots,n} \in \C^{m \times n}$ we denote by $A^T = (A_{jk})_{j =1,\ldots,n; k=1,\ldots,m} \in \C^{n\times m}$ the transpose of $A$. 
Moreover, we denote for every $d\in \N$ by $\Id_d \in \R^{d\times d}$ the $d$-dimensional identity matrix and by $\Sd$ the set of symmetric $d\times d$-matrices with entries from $\R$. 
For every $d \in \N$ and a $d$-dimensional continuous local martingale $M$ with i-th component process $M^{(i)}$, $i \in \{1,\ldots,d\}$, we denote by $\langle M \rangle = \sum_{i=1}^d \langle M^{(i)}\rangle$ the quadratic variation process of $M$. 
For $T \in (0,\infty)$, $m,n \in\N$, $f\colon [0,T] \to \C^{n\times m}$ and $h \in [0,T]$ the shifted function $\Delta_h f\colon [0,T] \to \C^{n \times m}$ satisfies $\Delta_h f(t) = f((t+h)\wedge T)$ for all $t \in [0,T]$. 
We denote for $T \in (0,\infty)$, $m,n \in \N$ and $f\colon [0,T] \to \C^{n \times m}$ the total variation of $f$ over $[0,T]$ by $\| f \|_{TV} = \sup \sum_{j} \| f(t_{j+1}) - f(t_j) \|$ where the supremum is taken over all partitions $0\leq t_1 <\ldots < t_N \leq T$, $N \in \N$. 
For $T \in (0,\infty)$, $d \in \N$ and a right-continuous function $f \colon [0,T] \to (\C^d)^*$ of bounded variation we denote its distributional derivative by $df$ and use the convention that $df(\{0\})=0$, such that  $f(t)=f(0)+\int_{[0,t]} df(s)$, $t \in [0,T]$.
For $T \in (0,\infty)$, a measurable function $F$ on $[0,T]$ and a measure $L$ on $[0,T]$ of bounded variation, both possibly matrix-valued\footnote{We say that $L$ is a $\C^{n\times m}$-valued measure on $[0,T]$ if each component of $L$ is a complex measure on $[0,T]$. Moreover, we say that $L$ is of bounded variation if each component has finite total variation on $[0,T]$. For further details we refer to \cite[Section 3.2]{Gripenberg1990Book}.}, we denote 
$(F \star L)(t) = \int_{[0,t]} F(t-s) L(ds)$, $(L \star F)(t) = \int_{[0,t]} L(ds) F(t-s)$ 
for $t \in (0,T]$, if these expressions are well-defined, and extend it to $t=0$ by right-continuity if possible. 
We also denote the convolution of two measurable functions $F,G$ on $[0,T]$ by $(F \star G)(t) = \int_0^t F(t-s)G(s)ds$, $t \in [0,T]$, if it is well-defined. 
For a square integrable function $F$ on $[0,T]$ and a continuous semimartingale $Z=\int_0^{\cdot} b_s ds + \int_0^{\cdot} \sigma_s dW_s$ with locally bounded adapted $b,\sigma$ and a Brownian motion $W$, both possibly multidimensional, we denote $(F\star dZ)(t) = \int_0^t F(t-s)dZ_s$, $t \in [0,T]$. 
For $T \in (0, \infty)$, $d \in \N$, $f \colon [0,T] \to \C$ and $\beta\ge 0$ let $I^\beta$ denote the Riemann–Liouville integral operator, i.e., $(I^\beta f)(t)=\frac{1}{\Gamma(\beta)}\int_0^t (t-s)^{\beta-1}f(s)ds$ for all $t \in [0,T]$ if $\beta>0$ and $(I^0 f)(t)=f(t)$ for all $t \in [0,T]$ whenever the expressions are well-defined. 
Moreover, let $D^\beta=\frac{d}{dt}I^{1-\beta}$ be the Riemann-Liouville fractional derivative.

\section{Moment bounds for stochastic inhomogeneous Volterra equations}
\label{sec:mom_bounds}
Let $T \in (0,\infty)$ and $m, d\in \N$. 
Let $b\colon [0,T] \times \R^d \to \R^d$, $\sigma\colon [0,T] \times \R^d \to \R^{d\times m}$ be two measurable functions, let $X_0 \in \R^d$ be a deterministic initial condition and 
let $K\colon [0,T]\times [0,T] \to \R^{d\times d}$ be
a given kernel.

We consider the $d$-dimensional stochastic inhomogeneous Volterra equation
\begin{equation}\label{eq:volterra_d_volt_inhom}
X_t=X_0+\int_0^tK(t,s)b(s,X_s)ds+\int_0^tK(t,s)\sigma(s,X_s)dW_s, \quad t \in [0,T],
\end{equation}
where $W=(W^{(1)}, W^{(2)}, \ldots, W^{(m)})^T$ is an $m$-dimensional Brownian motion.

\begin{defi} 
	\begin{enumerate}
		\item[(i)] 
		We say that the stochastic Volterra equation \eqref{eq:volterra_d_volt_inhom} has a strong solution if for any $m$-dimensional Brownian motion $W$ on a filtered probability space $(\Omega, \mathcal F, (\mathcal F_t)_{t\in [0,T]}, P)$ where $(\cF_t)_{t \in [0,T]}$ is the completed right-continuous filtration generated by $W$, there exists an adapted continuous $d$-dimensional process $X=(X_t)_{t\in[0,T]}$ such that \eqref{eq:volterra_d_volt_inhom} holds $P$-a.s.
		
		\item[(ii)] We say that the stochastic Volterra equation \eqref{eq:volterra_d_volt_inhom} has a weak solution if there exists a filtered probability space $(\Omega, \mathcal F, (\mathcal F_t)_{t\in [0,T]}, P)$ satisfying the usual conditions and supporting an $m$-dimensional $(\cF_t)_{t\in[0,T]}$-Brownian motion $W$ and an adapted continuous $d$-dimensional process $X=(X_t)_{t\in[0,T]}$ such that \eqref{eq:volterra_d_volt_inhom} holds $P$-a.s.
	\end{enumerate} 
\end{defi}

Lemma~\ref{lem:moments_of_soln_bdd} below provides moment bounds for continuous weak solutions of inhomogeneous stochastic Volterra equations. Its proof is based on ideas from \cite[Lemma 3.1]{Jaber2019affine} where moment bounds for stochastic Volterra equations with convolution kernel $K(t,s)=\ol K(t-s) 1_{s\leq t}$ and time-independent coefficients $b$ and $\sigma$ are established. We require the following two conditions.
\begin{assumption}\label{assumption:kernel_L_2}
	$K$ is measurable and it holds $\sup_{t \in [0,T]} \int_0^t \|K(t,s)\|^2 ds < \infty$. 
\end{assumption}
\begin{assumption}\label{assumption:linear_growth}
	There exists $c_{LG} \in (0,\infty)$ such that for all $t \in [0,T]$, $x \in \R^d$ it holds $\|b(t,x)\| + \| \sigma(t,x)\| \leq c_{LG}(1+\|x\|)$.
\end{assumption}

\begin{lemma}\label{lem:moments_of_soln_bdd}
	Assume that $b \in C([0,T]\times\R^d,\R^d)$ and $\sigma \in C([0,T]\times\R^d,\R^{d\times m})$. 
	Suppose that \ref{assumption:kernel_L_2} and \ref{assumption:linear_growth} are satisfied and that $X=(X_t)_{t\in[0,T]}$ is a continuous weak solution of \eqref{eq:volterra_d_volt_inhom}. 
	Then, for all $p\geq 1$ there exists a constant $c \in (0,\infty)$ depending only on  
	$T,d,m,\|X_0\|, K, c_{LG}, p$ such that 
	\begin{equation}\label{eq:moments_of_soln_bdd}
	\sup_{t \in [0,T]} E\left[ \|X_t\|^p \right] \leq c .
	\end{equation}
\end{lemma}

\begin{proof}	
	Throughout the proof let $p \in (2,\infty)$. The case $p \in [1,2]$ then follows from Jensen's inequality. 
	 	
	For each $n \in \N$ we define the stopping time $\tau_n = \inf\{ t\geq 0 \colon \| X_t\| \geq n \} \wedge T$. 
	As in \cite[Lemma 3.1]{Jaber2019affine} we observe that for all $n \in \N$, $t\in[0,T]$ we have 
	\begin{equation*}
	\begin{split}
	\| X_t \|^p 1_{\{t < \tau_n\}} & \leq \left\| X_0 + \int_0^t K(t,s) b(s, X_s 1_{\{s < \tau_n\}}) ds + \int_0^t K(t,s) \sigma(s,X_s1_{\{s<\tau_n\}}) dW_s \right\|^p .
	\end{split}
	\end{equation*}
	It hence holds for all $n \in \N$, $t \in [0,T]$ that 
	\begin{equation}\label{eq:17082020a1}
	\begin{split}
	E\left[ \| X_t\|^p 1_{\{t<\tau_n\}}  \right] 
	& \leq  3^{p-1} \Bigg( E\left[ \| X_0 \|^p \right] + E\left[ \left\| \int_0^t K(t,s)b(s,X_s1_{\{s<\tau_n\}}) ds  \right\|^p \right] \\
	& \quad  + E\left[ \left\| \int_0^t K(t,s)\sigma(s,X_s1_{\{s<\tau_n\}}) dW_s  \right\|^p \right] \Bigg) . 
	\end{split}
	\end{equation} 
	By Lemma~\ref{lem:fg_estimate_1} and Fubini's theorem we have that for all $n \in \N$, $t \in [0,T]$
	\begin{equation}\label{eq:17082020a2}
	\begin{split}
	& E\left[ \left\| \int_0^t K(t,s)b(s,X_s1_{\{s<\tau_n\}}) ds  \right\|^p \right] \\
	& \leq d^{\frac{3p}{2}} T^{\frac{p}{2}} 
	\left( \int_0^t \|K(t,s)\|^2 ds \right)^{\frac{p}{2}-1} 
	\int_0^t \| K(t,s) \|^2 E\left[ \| b(s,X_s1_{\{s<\tau_n\}}) \|^p  \right] ds .
	\end{split}
	\end{equation} 
	Applying successively Jensen's inequality, the Burkholder-Davis-Gundy inequality, Hölder's inequality (with $\frac{p-2}{p} + \frac{2}{p} = 1$) and Fubini's theorem yields that there exists $c_p \in (0,\infty)$ depending only on $p$ such that for all $n \in \N$, $t \in [0,T]$ 	 
	\begin{equation}\label{eq:17082020a3}
	\begin{split}
	& E\left[ \left\| \int_0^t K(t,s)\sigma(s,X_s1_{\{s<\tau_n\}}) dW_s  \right\|^p \right] \\
	& \leq E\left[ \sup_{r \in [0,t]} \left\| \int_0^{r\wedge t} K(t,s)\sigma(s,X_s1_{\{s<\tau_n\}}) dW_s  \right\|^p \right] \\
	& \leq d^{\frac{3p}{2}} m^{p} c_p  E\left[ \left( \int_0^{t} \| K(t,s) \|^{2-\frac{4}{p}} \| K(t,s) \|^{\frac{4}{p}}  \|\sigma(s,X_s1_{\{s<\tau_n\}})\|^2 ds \right)^{\frac{p}{2}} \right] \\
	& \leq d^{\frac{3p}{2}} m^{p} c_p   
	\left( \int_0^t \| K(t,s) \|^2 ds \right)^{\frac{p}{2}-1}
	\int_0^{t} \| K(t,s) \|^2  E\left[ \|\sigma(s,X_s1_{\{s<\tau_n\}}) \|^p \right] ds. 
	\end{split}
	\end{equation}
	Recall that \ref{assumption:kernel_L_2} ensures that $\sup_{t \in [0,T]} \int_0^t \|K(t,s)\|^2 ds < \infty$.
	Observe furthermore that \ref{assumption:linear_growth} implies that for all $n \in \N$, $s \in [0,T]$ we have $\| b(s,X_s1_{\{s<\tau_n\}}) \|^p + \|\sigma(s,X_s1_{\{s<\tau_n\}}) \|^p \leq 2^{p-1} c_{LG}^p (1 + \| X_s1_{\{s<\tau_n\}} \|^p )$.
	It therefore follows from \eqref{eq:17082020a1}, \eqref{eq:17082020a2} and \eqref{eq:17082020a3} that there exists $c_1 \in (0,\infty)$ depending only on $T,d,m,\|X_0\|, K, c_{LG}, p$ such that for all $n \in \N$, $t \in [0,T]$ 	 
	\begin{equation}\label{eq:estimate_for_EX_02}
	\begin{split}
	E\left[ \| X_t\|^p 1_{\{t<\tau_n\}}  \right] & \leq c_1 + c_1 \int_0^t \|K(t,s)\|^2 E\left[ \| X_s1_{\{s<\tau_n\}} \|^p \right] ds .
	\end{split}
	\end{equation}
	
	As in \cite[Lemma 3.1]{Jaber2019affine} we want to apply the generalized Gronwall lemma \cite[Lemma 9.8.2]{Gripenberg1990Book} with kernel $\widetilde K\colon [0,T]\times [0,T] \to \R$, $\widetilde K(t,s) = c_1 \| K(t,s)\|^2 1_{s\leq t}$. 
	Note that $\widetilde K$ is a scalar Volterra kernel (see \cite[Definition 9.2.1]{Gripenberg1990Book}). 
	Due to \ref{assumption:kernel_L_2} and \cite[Proposition 9.2.7(i)]{Gripenberg1990Book} we find that $\widetilde K$ is of type $L^{\infty}$ (see \cite[Definition 9.2.2]{Gripenberg1990Book}), and for all $0\leq t_1<t_{2}\leq T$ with $t_{2}-t_1$ sufficiently small it holds for the norm defined in \cite[Definition 9.2.2]{Gripenberg1990Book} that 
	\begin{equation*}
	\begin{split}
	||| \widetilde K |||_{L^{\infty}([t_1,t_{2}))} & \leq c_1 \sup_{t \in [t_1,t_{2})} \int_{t_1}^t \| K(t,s) \|^2 ds 
	< 1.
	\end{split}
	\end{equation*}
	Therefore, \cite[Corollary 9.3.14]{Gripenberg1990Book} implies that $- \widetilde K$ has a resolvent of type $L^{\infty}$ on $[0,T)$ (see \cite[Definition 9.3.1]{Gripenberg1990Book}), which we denote by $\widetilde R$. 
	By \cite[Proposition 9.8.1]{Gripenberg1990Book} and nonpositivity of $- \widetilde K$ we obtain that also $\widetilde R$ is nonpositive. 	
	Moreover, observe that for all $n \in \N$ it holds by definition of $\tau_n$ that 
	$\sup_{t \in [0,T)} E\left[ \| X_t\|^p 1_{\{t<\tau_n\}}  \right] \leq n^p$. 
	The generalized Gronwall lemma \cite[Lemma 9.8.2]{Gripenberg1990Book} 
	hence implies that for all $n \in \N$ and almost all $t \in [0,T)$ 
	\begin{equation*}
	E\left[ \| X_t\|^p 1_{\{t<\tau_n\}}  \right] \leq c_1 -	\int_0^t \widetilde R(t,s) c_1 ds. 
	\end{equation*}
	This yields for all $n \in \N$ that 
	\begin{equation*}
	\begin{split}
	\esssup_{t \in [0,T]} E\left[ \| X_t\|^p 1_{\{t<\tau_n\}}  \right] 
	& \leq c_1 + c_1 \esssup_{t \in [0,T]} \int_0^T \lvert \widetilde R(t,s) \rvert ds.
	\end{split}
	\end{equation*}
	Observe that \cite[Proposition 9.2.7(i)]{Gripenberg1990Book} shows that $\esssup_{t \in [0,T]} \int_0^T \lvert \widetilde R(t,s) \rvert ds < \infty$. 
	Hence, there exists $c_2 \in (0, \infty)$ which only depends on $T,d,m,\|X_0\|, K, c_{LG}, p$ such that it holds
	$\esssup_{t \in [0,T]} E\left[ \| X_t\|^p 1_{\{t<\tau_n\}}  \right] \leq c_2$. 
	This together with \eqref{eq:estimate_for_EX_02} and \ref{assumption:kernel_L_2} implies that also 
	\begin{equation*}
	\sup_{t \in [0,T]} E\left[ \| X_t\|^p 1_{\{t<\tau_n\}}  \right] \leq c
	\end{equation*}
	for some constant $c \in (0, \infty)$ that only depends on $T,d,m,\|X_0\|, K, c_{LG}, p$.
	
	By an application of Fatou's Lemma and taking the limit $n\to\infty$, we obtain that 
	\begin{equation*}
	\sup_{t\in [0,T]} E\left[ \| X_t\|^p \right] \leq c .
	\end{equation*}
\end{proof}

The following result is a simple consequence of \eqref{eq:volterra_d_volt_inhom} that is used in the proof of Theorem~\ref{thm:multdim_volt_kts} in Section~\ref{sec:fully_inhom_volt_case_main_thm}.

\begin{lemma}\label{lem:XcondF_simple_multidim_kts}
	Assume that $b \in C([0,T]\times\R^d,\R^d)$ and $\sigma \in C([0,T]\times\R^d,\R^{d\times m})$.
	Suppose that \ref{assumption:kernel_L_2} and \ref{assumption:linear_growth} are satisfied and that $X=(X_t)_{t\in[0,T]}$ is a continuous weak solution of \eqref{eq:volterra_d_volt_inhom}. 
	Then, for all $s,t\in [0,T]$ we have that
	\begin{equation}\label{eq:XcondF_simple_multidim_kts}
	E[X_s|\cF_t]=X_0+\int_0^sK(s,r)E[b(r,X_r)|\cF_t]dr+\int_0^{t\wedge s}K(s,r)\sigma(r,X_r)dW_r.
	\end{equation}
\end{lemma}

\begin{proof}
	If $s<t$ we have by \eqref{eq:volterra_d_volt_inhom}
	\begin{equation*}
	\begin{split}
	E[X_s|\cF_t]&=X_s
	=X_0+\int_0^sK(s,r)E[b(r,X_r)|\cF_t]dr+\int_0^{t\wedge s}K(s,r)\sigma(r,X_r)dW_r.
	\end{split}
	\end{equation*}
	For every $s \in [0,T]$
	let $M^s=(M^s_t)_{t\in[0,T]}$ be the process defined by $M^s_t=\int_0^{t \wedge s} K(s,r)\sigma(r,X_r)dW_r$, $t\in[0,T]$. 
	By assumptions \ref{assumption:kernel_L_2}, \ref{assumption:linear_growth} and Lemma~\ref{lem:moments_of_soln_bdd} we have 
	\begin{equation}\label{eq:Mistruemartingalesincequadraticvariation}
	\begin{split}
	E\left[ \langle M^s\rangle_T \right] & = E\left[ \int_0^{T \wedge s} \|K(s,r) \sigma(r,X_r)\|^2 dr \right] 
	\leq \left( \sup_{t \in [0,T]} E\left[ \| \sigma(t,X_t) \|^2 \right] \right) \int_0^s \| K(s,r)\|^2 dr < \infty.
	\end{split}
	\end{equation}
	For every $s \in [0,T]$ the process $M^s$ is hence a martingale.
	
	Therefore, in the case $s\ge t$ it holds by \eqref{eq:volterra_d_volt_inhom} that 
	\begin{equation*}
	\begin{split}
	E[X_s|\cF_t] & = X_0+E\left[\int_0^sK(s,r)b(r,X_r)dr \,\bigg|\,  \cF_t \right] + E\left[\int_0^sK(s,r)\sigma(r,X_r)dW_r\,\bigg|\, \cF_t\right]\\
	& = X_0 + \int_0^sK(s,r)E[b(r,X_r)|\cF_t]dr + \int_0^{t\wedge s}K(s,r)\sigma(r,X_r)dW_r.
	\end{split}
	\end{equation*} 
\end{proof}

\section{Conditional Fourier-Laplace functional of affine Volterra processes}\label{sec:fully_inhom_volt_case_main_thm}

In this section we 
consider affine stochastic Volterra processes, i.e., solutions $X$ of~\eqref{eq:volterra_d_volt_inhom} with some state space $E\subseteq \R^d$ where the coefficients have the following affine structure.
\begin{assumption}\label{assumption:affine}
Suppose there exist  
$b^{i}\in C([0,T],\R^d)$, $i\in\{0,\ldots,d\}$, such that 
\begin{equation*}
b(s,x) = b^0(s) + \sum_{i=1}^d b^{i}(s)x_i, \quad s \in [0,T],\, x \in \R^d.
\end{equation*}
\end{assumption}
Furthermore, assume there exist $E\subseteq \R^d$ and $A^{i}\in C([0,T],\Sd)$, $i\in\{0,\ldots,d\}$, such that for $a\colon [0,T] \times \R^d \to \Sd$ defined by 
\begin{equation*}
a(s,x) = A^0(s) + \sum_{i=1}^d A^{i}(s) x_i, \quad s \in [0,T],\, x \in \R^d, 
\end{equation*}
it holds 
\begin{equation*}
\sigma(s,x)\sigma(s,x)^T = a(s,x), \quad s \in [0,T],\, x \in E.
\end{equation*}

\medskip

Denote $B(s)=(b^1(s) \dots b^d(s)) \in \R^{d\times d}$, $s \in [0,T]$. Then, $b(s,x)=b^0(s) + B(s)x$ for all $s \in [0,T]$, $x \in \R^d$.  
Furthermore, for any $s \in [0,T]$ and any row vector $v \in (\C^d)^*$
we write $A(s,v)=(vA^1(s)v^T,\ldots,vA^d(s)v^T)$. 
Note that $va(s,x)v^T = vA^0(s)v^T + A(s,v)x$ , $s \in [0,T]$, $x \in \R^d$, $v \in (\C^d)^*$. 

Observe that \ref{assumption:affine} implies that 
$b$ is continuous and has at most linear growth and that $\sigma$ on $E$ has at most linear growth.

The following result shows that the conditional Fourier-Laplace functional of an affine Volterra process can be represented in terms of a solution to a Riccati-Volterra equation. 
It extends \cite[Theorem 4.3]{Jaber2019affine} to time-dependent coefficients $b$ and $\sigma$ and kernels $K$ that are not necessarily convolution kernels.

\begin{theo}\label{thm:multdim_volt_kts}
	Suppose that 
	\ref{assumption:affine} holds true and that there exists an $E$-valued continuous weak solution $X=(X_t)_{t \in [0,T]}$ of \eqref{eq:volterra_d_volt_inhom}.
	Assume that \ref{assumption:kernel_L_2} and  \ref{assumption:linear_growth} are satisfied and that $\sigma \in C([0,T]\times\R^d,\R^{d\times m})$.    
	Fix $u\in (\C^d)^*$ and $f \in L^1\left( [0,T], (\C^d)^* \right)$.
	Let $\psi\in L^2\left( [0,T], (\C^d)^* \right)$ solve the Riccati-Volterra equation 
	\begin{equation}\label{eq:riccati_multdim_kts}
	\begin{split}
	\psi(t)&=uK(T,t)+\int_t^{T} \left(f(s) + \psi(s)B(s)+\frac{1}{2}A(s,\psi(s))\right) K(s,t)ds, \quad t \in [0,T].
	\end{split}
	\end{equation}
	Then the $\C$-valued process $Y=(Y_t)_{t \in [0,T]}$ defined by
	\begin{equation}\label{eq:def_Y_kts}
	\begin{split}
	Y_t&=Y_0+\int_0^t\psi(s)\sigma(s,X_s)dW_s-\frac{1}{2}\int_0^t\psi(s)a(s,X_s)\psi(s)^Tds, \quad t \in [0,T],\\
	Y_0&=uX_0+\int_0^T f(s)X_0 + \psi(s)b(s,X_0) + \frac{1}{2}\psi(s)a(s,X_0)\psi(s)^T ds
	\end{split}
	\end{equation}
	satisfies
	\begin{equation}\label{eq:alt_rep_Y_multdim_kts}
	Y_t=E\left[uX_T + \int_0^T f(s)X_sds \,\bigg|\, \cF_t \right] + \frac{1}{2}\int_t^T\psi(s)a(s,E[X_s|\cF_t])\psi(s)^Tds, \quad t \in [0,T].
	\end{equation}
	The process $\exp(Y)$ is a local martingale and if it is a true martingale it holds
	\begin{equation}\label{eq:conditional_Fourier_Laplace_formula}
	E\left[\exp\left(uX_T + \int_0^T f(s) X_s ds \right)\,\bigg|\, \cF_t \right]=\exp(Y_t), \quad t\in[0,T].
	\end{equation}
\end{theo}

\begin{proof}
	As in the proof of \cite[Theorem 4.3]{Jaber2019affine} we 
	define the process $\widetilde Y$ by 
	\begin{equation*}
	\begin{split}
	\widetilde Y_t & = E\left[uX_T + \int_0^T f(s)X_sds \,\bigg|\, \cF_t\right] + \frac{1}{2}\int_t^T\psi(s)a(s,E[X_s|\cF_t])\psi(s)^Tds, \quad t \in [0,T],
	\end{split}
	\end{equation*}
	and we first show that $\widetilde Y_0 = Y_0$. 
	We have 
	\begin{equation*}
	\begin{split}
	\widetilde Y_0 - Y_0 
	& = uE[X_T-X_0 | \cF_0] + \int_0^T f(s) E[X_s|\cF_0] ds -\int_0^T f(s)X_0 ds - \int_0^T \psi(s)b(s,X_0)ds  \\
	& \quad + \frac{1}{2}\int_0^T\psi(s)a(s,E[X_s|\cF_0])\psi(s)^Tds 
	- \frac{1}{2} \int_0^T \psi(s)a(s,X_0)\psi(s)^T ds .
	\end{split}
	\end{equation*}
	Since $va(s,x)v^T = vA^0(s)v^T + A(s,v)x$ for all $s \in [0,T]$, $x \in \R^d$, $v \in (\C^d)^*$, it holds that 
	\begin{equation*}
	\begin{split}
	\int_0^T\psi(s)a(s,E[X_s|\cF_0])\psi(s)^Tds 
	- \int_0^T \psi(s)a(s,X_0)\psi(s)^T ds  
	& = \int_0^T A(s, \psi(s)) E[X_{s} - X_0|\cF_0] ds .
	\end{split}
	\end{equation*}
	Therefore, we obtain that 
	\begin{equation}\label{eq:ytilde0miny0multipre_kts}
	\begin{split}
	\widetilde Y_0 - Y_0 
	& = u E[X_T-X_0|\cF_0] + \int_0^T f(s) E[ X_{s} - X_0 |\cF_0]ds - \int_0^T \psi(s) b(s,X_0)ds \\
	& \quad + \frac12 \int_0^T A(s, \psi(s)) E[X_{s} - X_0|\cF_0] ds .
	\end{split}
	\end{equation}
	Moreover, since for any $s \in [0,T]$ the process $\int_0^{\cdot \wedge s} K(s,r)\sigma(r,X_r)dW_r$ is a true martingale (see~\eqref{eq:Mistruemartingalesincequadraticvariation}) and using Fubini's theorem, we have for all $t \in [0,T]$ that 
	\begin{equation}\label{eq:EXtminX0asConv_kts}
	\begin{split}
	E[X_t - X_0|\cF_0] & = E\left[ \int_0^t K(t,s)b(s,X_s)ds + \int_0^t K(t,s) \sigma(s,X_s)dW_s \,\bigg|\, \cF_0 \right]  \\
	& = \int_0^t K(t,s) E[b(s,X_s)|\cF_0] ds 
	= \int_0^t K(t,s) \left(b^0(s)+B(s)E[X_s|\cF_0]\right) ds \\ 
	& = \int_0^t K(t,s) b(s,E[X_s|\cF_0]) ds .
	\end{split}
	\end{equation}	
	Together with Fubini's theorem it follows that 
	\begin{equation*}
	\begin{split}
	& \int_0^T A(s,\psi(s)) E[X_s-X_0|\cF_0]ds 
	= \int_0^T A(s,\psi(s)) \left( \int_0^s K(s,r) b(r,E[X_r|\cF_0]) dr \right)  ds \\
	& = \int_0^T \int_0^s A(s,\psi(s)) K(s,r) b(r,E[X_r|\cF_0]) dr ds 
	= \int_0^T \int_r^T A(s,\psi(s)) K(s,r) b(r,E[X_r|\cF_0]) ds dr \\
	& = \int_0^T \int_r^{T} A(s,\psi(s)) K(s,r) ds \, b(r,E[X_r|\cF_0]) dr .
	\end{split}
	\end{equation*}
	The Riccati-Volterra equation~\eqref{eq:riccati_multdim_kts} then yields that 
	\begin{equation}\label{eq:kts001}
	\begin{split}
	& \frac12 \int_0^T A(s,\psi(s)) E[X_s-X_0|\cF_0]ds \\
	& = \int_0^T \left( \psi(r) - uK(T,r) 
	- \int_r^{T} \left( f(s) + \psi(s)B(s) \right) K(s,r)ds  \right) b(r,E[X_r |\cF_0])dr .
	\end{split}
	\end{equation}
	Observe furthermore that we have 
	\begin{equation}\label{eq:kts002}
	\begin{split}
	& \int_0^T \psi(s) b(s,X_0)ds 
	 = 
	\int_0^T \psi(s) b(s,E[X_s|\cF_0]) ds -\int_0^T \psi(s) B(s) E[X_s-X_0|\cF_0] ds
	\end{split}
	\end{equation}
	and by \eqref{eq:EXtminX0asConv_kts}
	\begin{equation}\label{eq:kts003}
	\begin{split}
	uE[X_T-X_0|\cF_0] & = u \int_0^T K(T,s) b(s,E[X_s|\cF_0]) ds .
	\end{split}
	\end{equation}
	By substituting \eqref{eq:kts001}, \eqref{eq:kts002} and \eqref{eq:kts003} into \eqref{eq:ytilde0miny0multipre_kts} we obtain 
	\begin{equation*}
	\begin{split}
	& \widetilde Y_0 -Y_0 \\
	& = u \int_0^T K(T,s) b(s,E[X_s|\cF_0]) ds  + \int_0^T f(s) E[ X_{s} - X_0 |\cF_0]ds \\
	& \quad - \int_0^T \psi(s) b(s,E[X_s|\cF_0]) ds + \int_0^T \psi(s) B(s) E[X_s-X_0|\cF_0] ds \\
	& \quad + \int_0^T \left( \psi(r) - uK(T,r) 
	- \int_r^{T} \left( f(s) + \psi(s)B(s) \right) K(s,r)ds  \right) b(r,E[X_r|\cF_0])dr \\
	& = \int_0^T \left(f(s)+\psi(s) B(s) \right) E[X_s-X_0|\cF_0] ds 
	 - \int_0^T \int_r^{T} \left( f(s) + \psi(s)B(s) \right) K(s,r)ds\, b(r,E[X_r|\cF_0])dr .
	\end{split}
	\end{equation*}
	It further follows from \eqref{eq:EXtminX0asConv_kts} and using Fubini's theorem that 
	\begin{equation*}
	\begin{split}
	\int_0^T \left(f(s)+\psi(s) B(s) \right) E[X_s-X_0|\cF_0] ds 
	& = \int_0^T \left( f(s)+\psi(s) B(s) \right) \left(  \int_0^s K(s,r) b(r,E[X_r|\cF_0]) dr \right) ds \\
	& = \int_0^T\int_r^T \left( f(s)+\psi(s) B(s) \right) K(s,r) b(r,E[X_r|\cF_0]) ds dr .
	\end{split}
	\end{equation*}	
	We have thus shown that $\widetilde Y_0 = Y_0$.  
	
	Next, in order to prove that $\widetilde Y = Y$, let $t \in [0,T]$ and observe that
	\begin{equation*}
	\begin{split}
	\widetilde Y_t & = E\left[uX_T + \int_0^T f(s)X_sds \,\bigg|\, \cF_t\right] + \frac{1}{2}\int_0^T \psi(s)a(s,E[X_s|\cF_t])\psi(s)^Tds \\
	& \quad - \frac{1}{2}\int_0^t \psi(s)a(s,E[X_s|\cF_t])\psi(s)^Tds \\
	& = uE[X_T|\cF_t] + \int_0^T f(s) E[X_s|\cF_t] ds + \frac{1}{2}\int_0^T \psi(s)A^0(s)\psi(s)^Tds \\
	& \quad + \frac12 \int_0^T A(s,\psi(s))E[X_s|\cF_t] ds - \frac{1}{2}\int_0^t \psi(s)a(s,X_s)\psi(s)^Tds .
	\end{split}
	\end{equation*}
	Let $C \in \C$ denote a quantity that does not depend on $t$ and note that we allow $C$ to change from line to line. 
	Then, we have 
	\begin{equation*}
	\begin{split}
	\widetilde Y_t - Y_t 
	& = C + uE[X_T|\cF_t] + \int_0^T f(s) E[X_s|\cF_t] ds + \frac12 \int_0^T A(s,\psi(s))E[X_s|\cF_t] ds \\
	& \quad  - \int_0^t\psi(s)\sigma(s,X_s)dW_s.
	\end{split}
	\end{equation*}
	It follows from Lemma~\ref{lem:XcondF_simple_multidim_kts} that 
	\begin{equation*}
	\begin{split}
	\widetilde Y_t - Y_t 
	& = C + \int_0^T f(s) E[X_s|\cF_t] ds + \int_0^T \left( uK(T,s)B(s) + \frac12 A(s,\psi(s)) \right) E[X_s|\cF_t] ds \\
	& \quad + \int_0^t \left(uK(T,s) - \psi(s) \right)\sigma(s,X_s)dW_s .
	\end{split}
	\end{equation*}
	We apply Lemma~\ref{lem:XcondF_simple_multidim_kts} once again to obtain 
	\begin{equation}\label{eq:tildeYtminYtmultidimpre_kts}
	\begin{split}
	\widetilde Y_t - Y_t 
	& = C + \int_0^T f(s) E[X_s|\cF_t] ds \\
	& \quad + \int_0^T \left( uK(T,s)B(s) + \frac12 A(s,\psi(s)) \right)
	\left( \int_0^s K(s,r)E[b(r,X_r)|\cF_t]dr \right) ds \\
	& \quad + \int_0^T \left( uK(T,s)B(s) + \frac12 A(s,\psi(s)) \right) \left(\int_0^{t\wedge s}K(s,r)\sigma(r,X_r)dW_r \right) ds \\
	& \quad + \int_0^t \left(uK(T,s) - \psi(s) \right)\sigma(s,X_s)dW_s .
	\end{split}
	\end{equation}
	Note that by the Riccati-Volterra equation~\eqref{eq:riccati_multdim_kts} it holds for all $s \in [0,t]$
	\begin{equation}\label{eq:riccati_rearranged_Tminr_multidim_ktr}
	\begin{split}
	uK(T,s) - \psi(s) 
	& = - \int_s^T  \left( f(r) + \psi(r) B(r) + \frac12 A(r,\psi(r)) \right) K(r,s) dr .
	\end{split}
	\end{equation}
	The stochastic Fubini theorem (see~\cite[Theorem 2.2]{Veraar2012stochFubini}) 
	implies that  
	\begin{equation}\label{eq:stochFubinikts}
	\begin{split}
	& \int_0^t \left(uK(T,s) - \psi(s) \right)\sigma(s,X_s)dW_s \\
	& = - \int_0^t \int_s^T  \left( f(r) + \psi(r) B(r) + \frac12 A(r,\psi(r)) \right) K(r,s) \sigma(s,X_s) dr dW_s \\
	& = - \int_0^T  \left( f(r) + \psi(r) B(r) + \frac12 A(r,\psi(r)) \right) \left( \int_0^{r\wedge t} K(r,s) \sigma(s,X_s) dW_s \right) dr.
	\end{split}
	\end{equation}
	To justify the application of this theorem, note that 
	for $\varphi\colon [0,T]\times [0,t] \times \R^d \to (\C^d)^*$, $\varphi(r,s,x) = \left( f(r) + \psi(r) B(r) + \frac12 A(r,\psi(r)) \right) K(r,s) \sigma(s,x) 1_{[0,r]}(s)$ with components $\varphi_k$, $k \in \{1,\ldots,m\}$, it holds 
	\begin{equation}\label{eq:estimateforjustofstochfubinikts}
	\begin{split}
	& \sum_{k=1}^m \left[ \int_0^T \left( \int_0^t (\Re \varphi_k(r,s,X_s))^2 ds \right)^{\frac12} dr + \int_0^T \left( \int_0^t (\Im \varphi_k(r,s,X_s))^2 ds \right)^{\frac12} dr \right] \\
	& \leq \sqrt{2m} \int_0^T \left( \int_0^t \| \varphi(r,s,X_s) \|^2 ds  \right)^{\frac12} dr \\
	& = \sqrt{2m} 
	 \int_0^T  \|f(r) + \psi(r) B(r) + \frac12 A(r,\psi(r))\| \left( \int_0^{t \wedge r} \|K(r,s)\|^2 \|\sigma(s,X_s)\|^2 ds \right)^{\frac12} dr \\
	& \leq \sqrt{2m} \left( \sup_{s \in [0,T]} \|\sigma(s,X_s)\| \right) \left( \sup_{r \in [0,T]} \int_0^r \|K(r,s)\|^2 ds \right)^{\frac12} \\
	& \quad \cdot \int_0^T \left( \|f(r)\| + \|\psi(r) B(r)\| + \frac12 \|A(r,\psi(r))\| \right) dr.
	\end{split}
	\end{equation}
	Since $\sigma$ is at most of linear growth and $X$ is continuous we obtain that $\sup_{s \in [0,T]} \|\sigma(s,X_s)\| < \infty$. 
	Furthermore, $\sup_{r \in [0,T]} \int_0^r \|K(r,s)\|^2 ds$ is finite by assumption \ref{assumption:kernel_L_2}. 
	Moreover, since $b^1,\ldots,b^d$ and $A^1,\ldots,A^d$ are continuous and hence bounded on $[0,T]$ (recall \ref{assumption:affine}), there exists $c \in (0,\infty)$ such that for all $r \in [0,T]$, $\|\psi(r)B(r)\| \leq  c  \|\psi(r)\|$ and $\|A(r,\psi(r))\| \leq c \|\psi(r)\|^2$. Now, we conclude from $f \in L^1\left( [0,T], (\C^d)^* \right)$ and $\psi\in L^2\left( [0,T], (\C^d)^* \right)$ that \eqref{eq:estimateforjustofstochfubinikts} is finite. We thus can apply~\cite[Theorem 2.2]{Veraar2012stochFubini} to each $\R$-valued integral $\int_0^t \int_0^T \Re \varphi_k(r,s,X_s)drdW_s^{(k)}$, $\int_0^t \int_0^T \Im \varphi_k(r,s,X_s)drdW_s^{(k)}$, $k \in \{1,\ldots,m\}$.
		
	Combining \eqref{eq:stochFubinikts} and  \eqref{eq:tildeYtminYtmultidimpre_kts} implies that 
	\begin{equation*}
	\begin{split}
	 \widetilde Y_t - Y_t 
	& = C + \int_0^T f(s) E[X_s|\cF_t] ds \\
	& \quad + \int_0^T \left( uK(T,s)B(s) + \frac12 A(s,\psi(s)) \right)
	\left( \int_0^s K(s,r)E[b(r,X_r)|\cF_t]dr \right) ds \\
	& \quad + \int_0^T \left( uK(T,s)B(s) -  f(s) - \psi(s) B(s) \right) \left(\int_0^{t\wedge s}K(s,r)\sigma(r,X_r)dW_r \right) ds .
	\end{split}
	\end{equation*}
	Observe that by Lemma~\ref{lem:XcondF_simple_multidim_kts} it holds for all $s \in [0,T]$
	\begin{equation*}
	\begin{split}
	\int_0^{t\wedge s}K(s,r)\sigma(r,X_r)dW_r 
	& = E[X_s | \cF_t ] - X_0 - \int_0^s K(s,r) E[b(r,X_r)|\cF_t] dr
	\end{split}
	\end{equation*}
	which implies 
	\begin{equation}\label{eq:kts004}
	\begin{split}
	\widetilde Y_t - Y_t  
	& = C + \int_0^T f(s) E[X_s|\cF_t] ds \\
	& \quad + \int_0^T \left( \frac12 A(s,\psi(s)) + f(s) + \psi(s) B(s) \right)
	\left( \int_0^s K(s,r)E[b(r,X_r)|\cF_t]dr \right) ds \\
	& \quad + \int_0^T \left( uK(T,s)B(s) -  f(s) - \psi(s) B(s) \right) E[X_s | \cF_t ] ds \\
	& = C + \int_0^T \left( \frac12 A(s,\psi(s)) + f(s) + \psi(s) B(s) \right)
	\left( \int_0^s K(s,r) B(r)E[X_r|\cF_t] dr \right) ds \\
	& \quad + \int_0^T \left( uK(T,s)B(s) - \psi(s) B(s) \right) E[X_s | \cF_t ] ds.	
	\end{split}
	\end{equation}	
	It further holds, using Fubini's theorem and subsequently the Riccati-Volterra  equation~\eqref{eq:riccati_multdim_kts}, that 
	\begin{equation}\label{eq:kts005}
	\begin{split}
	& \int_0^T \left( \frac12 A(s,\psi(s)) + f(s) + \psi(s) B(s) \right)
	\left( \int_0^s K(s,r) B(r)E[X_r|\cF_t] dr \right) ds \\
	& = \int_0^T \left( \int_r^T \left( \frac12 A(s,\psi(s)) + f(s) + \psi(s) B(s) \right) K(s,r) ds \right) 
	 B(r)E[X_r|\cF_t] dr \\
	& = \int_0^T \left( \psi(r) - uK(T,r) \right) B(r)E[X_r|\cF_t] dr .
	\end{split}
	\end{equation}
	It follows from \eqref{eq:kts004} and \eqref{eq:kts005} that $\widetilde Y_t -Y_t = C$.  
	The fact that $\widetilde Y_0 = Y_0$ implies that $C = 0$ and hence $\widetilde Y = Y$. 
	
	For the remaining claims, note that by~\eqref{eq:def_Y_kts}, $Y+\frac12 \langle Y \rangle$ is a local martingale and hence $\exp(Y)$ is a local martingale. 
	Since $Y_T = uX_T + \int_0^T f(s)X_sds$ (cf.~\eqref{eq:alt_rep_Y_multdim_kts}), we have $E\left[ \exp\left( uX_T + \int_0^T f(s)X_sds \right) \,\Big|\,  \cF_t \right] = \exp(Y_t)$ for all $t \in [0,T]$ if $\exp(Y)$ is a true martingale.
\end{proof}

In the special case where $K=\Id_d$ and where the coefficients $b$ and $\sigma$ are time-independent and exhibit an affine structure as in assumption~\ref{assumption:affine}, the solution $X$ of \eqref{eq:volterra_d_volt_inhom} is an affine diffusion process in the classical sense (see, e.g., \cite{filipovic2009affine} and \cite{Duffie2003affineProcFinance}). The conditional Fourier-Laplace transform of such an affine process, under appropriate regularity conditions, has an exponential-affine structure $\exp(\phi(T-t)+\chi(T-t)X_t)$, $t \in [0,T]$, where $\phi$ and $\chi$ are determined by the solution of a Riccati equation
(see e.g. \cite[Definition 2.1 and Theorem 2.2]{filipovic2009affine}). 
For affine stochastic Volterra processes with convolution kernels $K(t,s)=\ol K(t-s)1_{s\leq t}$, $s,t \in [0,T]$, and affine time-independent coefficients $b$ and $\sigma$ it is shown in \cite{Jaber2019affine} that at time zero such an exponential-affine structure  
$\exp(\phi(T)+\chi(T)X_0)$ holds true as well (see the discussion after \cite[Theorem 4.3]{Jaber2019affine}).  
Theorem \ref{thm:multdim_volt_kts} above implies that also for an inhomogeneous kernel $K$ and time-dependent coefficients $b$ and $\sigma$, the Fourier-Laplace transform at time zero still has this form.  
This is the content of the following remark. 
\begin{remark}\label{rem:def_phi_chi}
	Under the assumptions and with the notation of Theorem~\ref{thm:multdim_volt_kts} it holds for 
	$\phi\colon [0,T] \to \C$ and $\chi\colon [0,T]\to (\C^d)^*$ defined by 
	\begin{equation}\label{eq:def_phi_chi}
	\begin{split}
	\phi(t)&=\int_{T-t}^{T} \psi(s)b^0(s) + \frac{1}{2}\psi(s)A^0(s)\psi(s)^Tds, \quad t\in [0,T],\\
	\chi(t)&=u+\int_{T-t}^T f(s) + \psi(s)B(s) + \frac{1}{2}A(s,\psi(s))ds, \quad t\in [0,T],
	\end{split}
	\end{equation}
	that 
	\begin{equation}\label{eq:Y_0_in_terms_of_phi_chi}
	Y_0 = \phi(T) + \chi(T)X_0.
	\end{equation}
	It follows that if $e^Y$ is a true martingale, then   
	\begin{equation}\label{eq:time_zero}
	E\left[\exp\left(uX_T + \int_0^T f(s)X_sds\right)\right] 
	=\exp(\phi(T)+\chi(T)X_0).
	\end{equation}
\end{remark}

\section{Volterra processes with convolution kernel}
\label{sec:conv_kernel}
Throughout this section, we assume that the kernel $K \colon [0,T] \times [0,T] \to \R^{d\times d}$ in \eqref{eq:volterra_d_volt_inhom} is a convolution kernel, i.e., there exists $\ol K\colon [0,T] \to \R^{d\times d}$ such that $K(t,s)=\ol K(t-s) 1_{s\leq t}$ for all $s,t \in [0,T]$.
The stochastic Volterra equation~\eqref{eq:volterra_d_volt_inhom} thus becomes 
\begin{equation}\label{eq:time_dep_volterra_eq}
X_t=X_0+\int_0^t \ol K(t-s)b(s,X_s)ds+\int_0^t \ol K(t-s)\sigma(s,X_s)dW_s, \quad t\in [0,T].
\end{equation}

In the case of time-independent coefficients $b$ and $\sigma$ the paper \cite{Jaber2019affine} establishes existence results for \eqref{eq:time_dep_volterra_eq}.
In Subsection~\ref{sec:existence_of_soln_of_convolution_SDE} we apply these results to the time-space process $(t,X_t)^T$, $t \in [0,T]$, to obtain existence results for \eqref{eq:time_dep_volterra_eq} also in the setting with time-dependent $b$ and $\sigma$. 
Subsection~\ref{sec:represent_for_Y} combines the convolution case with the affine setting from Section~\ref{sec:fully_inhom_volt_case_main_thm}. In Theorem~\ref{thm:represent_for_Y} we establish that $Y$ from Theorem~\ref{thm:multdim_volt_kts} is an affine function of $X$. This generalizes \cite[Theorem 4.5]{Jaber2019affine} to time-dependent coefficients. 

We first draw some consequences from Theorem~\ref{thm:multdim_volt_kts} for the special case of a convolution kernel $K(t,s)=\ol K(t-s) 1_{s\leq t}$, $s,t \in [0,T]$.
\begin{defi}\label{rem:intro_resolvent_of_first_kind}
	$\ol K$ is said to admit a resolvent of the first kind if there exists an $\R^{d\times d}$-valued measure $\ol L$ on $[0,T]$ which is of bounded variation and satisfies $\ol K\star \ol L= \ol L\star \ol K=\Id_d$ (see also \cite[Definition 5.5.1]{Gripenberg1990Book}). 
\end{defi}
	
	There is at most one resolvent of the first kind for a given kernel $\ol K \in L^1\left( [0,T], \R^{d\times d} \right)$ (see \cite[Theorem 5.5.2]{Gripenberg1990Book}). Moreover,
	if a scalar kernel $k \in L^1([0,T],[0,\infty))$ is not identically zero and nonincreasing, it follows from \cite[Theorem 5.5.5]{Gripenberg1990Book} that $k$ possesses a resolvent of the first kind $L$ and $L$ is nonnegative. 

The next remark provides a reformulation of the Riccati-Volterra equation  \eqref{eq:riccati_multdim_kts}, if $\ol K$ admits a resolvent of the first kind.
\begin{remark}\label{rem:affine_volt_proc_conv}
	Suppose that $K(t,s)=\ol K(t-s) 1_{s\leq t}$, $s,t \in [0,T]$, and that \ref{assumption:kernel_L_2} and \ref{assumption:affine} for some $E\subseteq\R^d$ are satisfied,  
	and fix $u\in (\C^d)^*$, $f \in L^1\left( [0,T], (\C^d)^* \right)$. 		
	If $\ol K$ admits a resolvent of the first kind $\ol L$,  
	then the Riccati-Volterra equation  \eqref{eq:riccati_multdim_kts} is equivalent to 
	\begin{equation}
	\begin{split}
		\int_{[0,T-t]} \psi(t+s) \ol L(ds) & = u + \int_{t}^T f(s) + \psi(s)B(s) + \frac12 A(s,\psi(s))ds, \quad t\in [0,T].
	\end{split}
	\end{equation}
	Using the notation for convolutions, this can be written as 
	\begin{equation}\label{eq:riccati_multdim_kts_tilde_psi_star_L}
	\begin{split}
	(\psi(T-\cdot) \star \ol L)(t) & = u + \int_{T-t}^T f(s) + \psi(s)B(s) + \frac12 A(s,\psi(s))ds, \quad t\in [0,T].
	\end{split}
	\end{equation}
\end{remark}

We next examine further the subcase of a fractional kernel.
\begin{remark}\label{rem:frac_kernel}
	Let $\alpha_1,\ldots,\alpha_d\in (\frac{1}{2},1]$ and 
	let $\ol K=\diag(\ol K_1,\ldots,\ol K_d)$ with $\ol K_i(t)=\frac{t^{\alpha_i-1}}{\Gamma(\alpha_i)}$, $i\in \{1,\ldots,d\}$. 
	Define $\ol L=\diag(\ol L_1,\ldots,\ol L_d)$ to be the $\R^{d\times d}$-valued measure on $[0,T]$ which satisfies $\ol L_i(dt)=\frac{t^{-\alpha_i}}{\Gamma(1-\alpha_i)}dt$ if $\alpha_i<1$ and $\ol L_i(dt)=\delta_0(dt)$ if $\alpha_i=1$. 
	Then it holds that $\ol K\star \ol L = \ol L \star \ol K = \Id_d$ and hence $\ol L$ is the resolvent of the first kind of $\ol K$.

	Observe that for any $F \in L^2\left( [0,T], (\C^d)^* \right)$ it holds 
	\begin{equation}\label{eq:tilde_psi_star_L}
	(F \star \ol L)(t)=\left(I^{1-\alpha_1}F_1(t),\ldots,I^{1-\alpha_d}F_d(t) \right)=:(I^{1-\alpha} F)(t),
	\end{equation}
	where $I^\beta$ for $\beta \geq 0$ denotes the Riemann–Liouville integral operator. Recall also that for $\beta \geq 0$, $D^\beta=\frac{d}{dt}I^{1-\beta}$ denotes the Riemann-Liouville fractional derivative. 
	Let \ref{assumption:affine} hold true for some $E\subseteq \R^d$ 
	and fix $u\in (\C^d)^*$, $f \in L^1\left( [0,T], (\C^d)^* \right)$. 
	In view of \eqref{eq:tilde_psi_star_L} and  \eqref{eq:riccati_multdim_kts_tilde_psi_star_L} we have that $\psi \in L^2([0,T],(\C^d)^*)$ solves the  Riccati-Volterra equation \eqref{eq:riccati_multdim_kts} if and only if for all $i\in \{1,\ldots,d\}$, $t\in [0,T]$, it holds that
	\begin{equation}\label{eq:equi_riccati_frac}
	(D^{\alpha_i}\widetilde \psi_i)(t)
	=  f_i(T-t) + \widetilde \psi(t) b^i(T-t) + \frac{1}{2} \widetilde \psi(t) A^i(T-t) \widetilde \psi(t)^T, \quad (I^{1-\alpha_i} \widetilde \psi_i)(0)=u_i, 
	\end{equation}
	where $\widetilde\psi = \psi(T-\cdot)$.
	
	Suppose now that $\widetilde\psi \in L^2([0,T],(\C^d)^*)$ solves \eqref{eq:equi_riccati_frac}. Furthermore, assume that $\sigma \in C([0,T]\times\R^d,\R^{d\times m})$, that \ref{assumption:linear_growth} holds, that there exists an $E$-valued continuous weak solution $X=(X_t)_{t \in [0,T]}$ of \eqref{eq:time_dep_volterra_eq} and that $e^Y$ with $Y$ defined by \eqref{eq:def_Y_kts} is a true martingale. 
	Then, with $\phi$ defined by \eqref{eq:def_phi_chi}, it follows from \eqref{eq:time_zero}, \eqref{eq:riccati_multdim_kts_tilde_psi_star_L} and \eqref{eq:tilde_psi_star_L} that
	\begin{equation}\label{eq:time_zero_frac}
	E\left[\exp\left(uX_T + \int_0^T f(s)X_sds \right) \right] = 
	\exp\left(\phi(T)+(I^{1-\alpha} \widetilde \psi)(T) X_0 \right).
	\end{equation}
\end{remark}

\subsection{Existence of solutions in the convolution case}\label{sec:existence_of_soln_of_convolution_SDE}

In the present subsection we make the following assumption, where $\ol K_{ij}\colon [0,T] \to \R$, $i,j \in \{1,\ldots,d\}$, denote the components of the 
kernel $\ol K \colon [0,T] \to \R^{d\times d}$:
\begin{assumption}\label{assumption:kernel_order}
	There exist $\gamma \in (0,2]$ and $c \in (0,\infty)$ such that for all $i,j \in \{1,\ldots,d\}$ it holds $\int_s^t \left(\ol K_{ij}(t-u)\right)^2 du \leq c (t-s)^{\gamma}$ and $\int_0^s \left( \ol K_{ij}(t-u) - \ol K_{ij}(s-u) \right)^2 du \leq c (t-s)^{\gamma}$ for all $0\leq s< t\leq T$.
\end{assumption}
Assumption~\ref{assumption:kernel_order} guarantees that processes of the form $\int_0^{\cdot} \ol K(\cdot-s) \ol b_s ds + \int_0^{\cdot} \ol K(\cdot-s) \ol a_s dW_s$ have a Hölder continuous version under appropriate assumptions on the processes $\ol a$ and $\ol b$ (cf. \cite[Lemma 2.4]{Jaber2019affine}).
This result plays an important role in the proofs of the existence results (Theorems 3.3, 3.4 and 3.6)  in \cite{Jaber2019affine}.
Observe that for a convolution kernel, assumptions \ref{assumption:kernel_L_2} and \ref{assumption:kernel_order} together are the same as condition (2.5) in \cite{Jaber2019affine}.

In order to apply results from \cite{Jaber2019affine} we introduce the following coefficients that do not depend on time: 
Let $\widetilde b\colon \R^{d+1} \to \R^{d+1}$ be defined by $\widetilde b_1(x)=1$, $x \in \R^{d+1}$, and for $j \in \{2,\ldots d+1\}$
\begin{equation*}
\widetilde{b}_j(x) = 
\begin{cases}
b_{j-1}(x), & x \in [0,T] \times \R^d, \\
b_{j-1}(0,(x_2,\ldots,x_{d+1})^T), & x \in (-\infty,0) \times \R^d,\\
b_{j-1}(T,(x_2,\ldots,x_{d+1})^T), & x \in (T,\infty) \times \R^d.
\end{cases}
\end{equation*}
Furthermore, let $\widetilde \sigma\colon \R^{d+1} \to \R^{(d+1)\times m}$ be defined by $\widetilde \sigma_{1,k}(x)=0$, $x \in \R^{d+1}$, $k \in \{1,\ldots,m\}$, and for $j \in \{2,\ldots d+1\}$, $k \in \{1,\ldots,m\}$
\begin{equation*}
\widetilde{\sigma}_{j,k}(x) = 
\begin{cases}
\sigma_{j-1,k}(x), & x \in [0,T] \times \R^d, \\
\sigma_{j-1,k}(0,(x_2,\ldots,x_{d+1})^T), & x \in (-\infty,0) \times \R^d,\\
\sigma_{j-1,k}(T,(x_2,\ldots,x_{d+1})^T), & x \in (T,\infty) \times \R^d.
\end{cases}
\end{equation*}
We also define a convolution kernel $\widetilde K \colon [0,T] \to \R^{(d+1)\times (d+1)}$ by  
$\widetilde{K}_{i,j}(t) = \ol K_{i-1,j-1}(t)$, $i,j \in\{2,\ldots,d+1\}$, and $\widetilde K_{1,1}(t)=1$ and 
$\widetilde{K}_{1,j}(t)=0 = \widetilde{K}_{j,1}(t)$, $j\in\{2,\ldots,d+1\}$ for all $t \in [0,T]$. 
Moreover, let $\widetilde X_0 = (0,X_0)^T$. 
We then consider the $(d+1)$-dimensional homogeneous stochastic Volterra equation 
\begin{equation}\label{eq:volterra_eq_hom_coeff}
\widetilde X_t = \widetilde X_0 + \int_0^t \widetilde  K(t-s) \widetilde b(\widetilde X_s)ds + \int_0^t \widetilde K(t-s)\widetilde \sigma(\widetilde X_s)dW_s, \quad t\in [0,T].
\end{equation}

Note that if $\ol K$ satisfies \ref{assumption:kernel_L_2} and \ref{assumption:kernel_order}, the same holds true for $\widetilde{K}$.  
Furthermore, if $\ol K$ admits a resolvent of the first kind $\ol L$ (see Definition~\ref{rem:intro_resolvent_of_first_kind}), we can obtain a resolvent of the first kind for $\widetilde K$ by setting $\widetilde{L}_{1,1}=\delta_0$, $\widetilde L_{1,j} \equiv 0 \equiv \widetilde{L}_{j,1}$, $j \in \{2,\ldots,d+1\}$, and $\widetilde{L}_{i,j} = \ol L_{i-1,j-1}$, $i,j \in \{2,\ldots,d+1\}$. 
Observe furthermore that Lipschitz continuity of $b, \sigma$ passes on to $\widetilde b, \widetilde \sigma$. 
Moreover, if $b,\sigma$ are continuous and satisfy the linear growth condition \ref{assumption:linear_growth}, then also $\widetilde{b}, \widetilde{\sigma}$ are continuous with \ref{assumption:linear_growth}. 

Note that if there exists a continuous solution $\widetilde{X}$ of \eqref{eq:volterra_eq_hom_coeff}, its first component satisfies $\widetilde{X}_t^{(1)}=t$ for all $t \in [0,T]$. Therefore, we have for all $j \in \{2,\ldots,d+1\}$, $k \in \{1,\ldots,m\}$ that $\widetilde b_j(\widetilde X_s) = b_{j-1}(s, (\widetilde{X}_s^{(2)}, \ldots, \widetilde{X}_s^{(d+1)})^T)$ 
and 
$\widetilde \sigma_{j,k}(\widetilde X_s) = \sigma_{j-1,k}(s, (\widetilde{X}_s^{(2)}, \ldots, \widetilde{X}_s^{(d+1)})^T)$ 
for all $s \in [0,T]$. 
This implies that $(\widetilde{X}^{(2)}, \ldots, \widetilde{X}^{(d+1)})^T$ is a continuous solution of \eqref{eq:volterra_d_volt_inhom}.

\begin{propo}\label{propo:existence_strong_soln_sde_conv_Lip}
	Assume that $K(t,s)=\ol K(t-s) 1_{s\leq t}$, $s,t \in [0,T]$, and suppose that \ref{assumption:kernel_L_2} and \ref{assumption:kernel_order} are satisfied. 
	Let the coefficients $b,\sigma$ be Lipschitz continuous. 
	Then \eqref{eq:time_dep_volterra_eq} admits a unique continuous strong solution.
\end{propo}

\begin{proof}
	This holds true by \cite[Theorem 3.3]{Jaber2019affine} and the arguments preceding Proposition~\ref{propo:existence_strong_soln_sde_conv_Lip}.
\end{proof}

\begin{propo}\label{propo:existence_weak_soln_sde_conv}
	Assume that $K(t,s)=\ol K(t-s) 1_{s\leq t}$, $s,t \in [0,T]$, and suppose that \ref{assumption:kernel_L_2} and \ref{assumption:kernel_order} are satisfied. 
	Furthermore, assume that $\ol K$ admits a resolvent of the first kind $\ol L$. 	
	Suppose that $b \in C([0,T]\times\R^d,\R^d)$ and $\sigma \in C([0,T]\times\R^d,\R^{d\times m})$ satisfy \ref{assumption:linear_growth}.  
	Then \eqref{eq:time_dep_volterra_eq} admits a continuous weak solution.
\end{propo}

\begin{proof}
	This follows from \cite[Theorem 3.4]{Jaber2019affine} and the arguments preceding Proposition~\ref{propo:existence_strong_soln_sde_conv_Lip}.
\end{proof}

\begin{propo}\label{propo:existence_weak_soln_sde_conv_nonneg}
	Assume that $K(t,s)=\ol K(t-s) 1_{s\leq t}$, $s,t \in [0,T]$, is diagonal and for each $i \in \{1,\ldots,d\}$, $\ol K_{ii}$ satisfies \ref{assumption:kernel_L_2} and \ref{assumption:kernel_order} and is nonnegative, not identically zero, nonincreasing and continuous on $(0,T]$, and its resolvent of the first kind $\ol L_{ii}$ is nonnegative and for all $t \in [0,T]$, the function $[0,T-t] \ni s \mapsto \ol L_{ii}([s,s+t])$ is nonincreasing. 	
	Furthermore, assume that $b \in C([0,T]\times\R^d,\R^d)$ and $\sigma \in C([0,T]\times\R^d,\R^{d\times m})$ satisfy \ref{assumption:linear_growth} and that for all $t \in [0,T]$, $i \in \{1,\ldots,d\}$, $k \in\{1,\ldots,m\}$ and all $x \in \R^d$ with $x_i=0$ it holds $b_i(t,x) \geq 0$ and $\sigma_{i,k}(t,x) = 0$. 	
	Assume that $X_0 \in [0,\infty)^d$. 
	Then \eqref{eq:time_dep_volterra_eq} admits a $[0,\infty)^d$-valued continuous weak solution.
\end{propo}

\begin{proof}
	In this setting, similarly as above, the assumptions on $\ol K$, $\ol L$, $b$, $\sigma$ and $X_0$ imply that $\widetilde{K}$, $\widetilde{L}$, $\widetilde{b}$, $\widetilde{\sigma}$ and $\widetilde X_0$ satisfy the conditions of \cite[Theorem 3.6]{Jaber2019affine}. This yields the claim.
\end{proof}

\subsection{Affine Volterra processes with convolution kernel}\label{sec:represent_for_Y}

In this subsection we consider stochastic Volterra equations with a convolution kernel, where the coefficients moreover possess an affine structure as in \ref{assumption:affine}.

Recall that the conditional Fourier-Laplace transform at time $t \in [0,T]$ of a classical affine diffusion process $X$ is exponential-affine in $X_t$ (see the discussion after Theorem~\ref{thm:multdim_volt_kts}). 
This property does not 
generalize to affine Volterra processes since these processes are typically not Markovian. 
In \cite[Theorem 4.3 and Theorem 4.5]{Jaber2019affine} it is shown that the conditional Fourier-Laplace transform at time $t \in [0,T]$ of an affine stochastic Volterra process $X$ with convolution kernel $K(t,s)=\ol K(t-s)1_{s\leq t}$, $s,t\in[0,T]$, and time-independent coefficients $b$ and $\sigma$ has an exponential-affine structure in the past trajectory $(X_s)_{s\in[0,t]}$.
Theorem~\ref{thm:represent_for_Y} below combined with Theorem~\ref{thm:multdim_volt_kts} extends this result to the case of time-dependent coefficients. 
Representation~\eqref{eq:repres_Y_conv_case} for $Y$ in Theorem~\ref{thm:represent_for_Y} below is convenient for proving that in the inhomogeneous Volterra-Heston model (see Section~\ref{sec:inhom_volt_heston_model}), $e^Y$ is a true martingale (see Proposition~\ref{propo:exp_Y_is_true_martingale_heston}).

\begin{theo}\label{thm:represent_for_Y}
	Let $K(t,s)=\ol K(t-s) 1_{s\leq t}$, $s,t \in [0,T]$. 
	Suppose that 
	\ref{assumption:affine} holds true and that there exists an $E$-valued continuous weak solution $X=(X_t)_{t \in [0,T]}$ of  \eqref{eq:time_dep_volterra_eq}. 
	Assume that \ref{assumption:kernel_L_2} and  \ref{assumption:linear_growth} are satisfied and that $\sigma \in C([0,T]\times\R^d,\R^{d\times m})$. 
	Furthermore, assume that $\ol K$ is continuous on $(0,T]$ and admits a resolvent of the first kind $\ol L$ and that $\sup_{r \leq T} \| (\Delta_r \ol K) \star \ol L \|_{TV} < \infty$. 
	Fix $u\in (\C^d)^*$ and $f \in L^1\left( [0,T], (\C^d)^* \right)$.
	Let $\psi\in L^2\left( [0,T], (\C^d)^* \right)$ solve the Riccati-Volterra equation 
	\eqref{eq:riccati_multdim_kts} and let $Y$ be defined by \eqref{eq:def_Y_kts}. 
	\begin{itemize}
		\item[(i)]  
		It holds that $\psi(t-\cdot)\star \ol L \colon [0,t] \to (\C^d)^*$ is right-continuous and of bounded variation, and that 	
		\begin{equation}\label{eq:repres_Y_cov_case_early}
			\begin{split}
				Y_t & = \left( \int_{[0,T]} \psi(s) \ol L(ds) \right) X_0 + 
				\psi(t) \ol L(\{0\})X_t - 
				\left( \int_{[0,t]} \psi(s) \ol L(ds) \right) X_0 \\
				& \quad + \int_{[0,t]} d\left( \psi(t-\cdot) \star \ol L \right)(s) X_{t-s}  
				+ \int_t^T \psi(s) b^0(s) ds \\
				& \quad + \frac12 \int_t^T \psi(s) A^0(s) \psi(s)^T ds 
				- \int_0^t \left( \psi(s) B(s) + \frac12 A(s,\psi(s)) \right)X_s ds, \quad t \in [0,T].
			\end{split}
		\end{equation}
	
		\item[(ii)] 			  
		For all $t \in [0,T]$ let
		\begin{equation}\label{eq:def_pi_h}
			g_t \colon [0,t] \to (\C^d)^*, \quad 
			g_t(r) = - \int_{(r,T-t+r]} \psi(t-r+s) \ol L(ds), \quad r \in [0,t].
		\end{equation}
		It then holds for all $t \in [0,T]$ that $g_t\colon [0,t] \to (\C^d)^*$ is right-continuous and of bounded variation, and that  
		\begin{equation}\label{eq:repres_Y_conv_case}
			\begin{split}
				Y_t & = -g_t(t) X_0 + \psi(t) \ol L(\{0\})X_t +  \int_{[0,t]} (dg_t(s)) X_{t-s} + \int_0^t f(s)X_sds \\
				& \quad 
				+ \int_t^T \psi(s)b^0(s)ds + \frac12 \int_t^T \psi(s) A^0(s) \psi(s)^T ds, \quad t \in [0,T].
			\end{split}
		\end{equation}
	\end{itemize}
\end{theo}

\begin{proof}
1. Note first that for all $t \in [0,T]$ it holds  
\begin{equation}\label{eq:g_t_as_conv}
g_t = \psi(t-\cdot) \star \ol L - (\psi(T-\cdot) \star \ol L)(T-t+\cdot),
\end{equation}
	and that we prove in part 2 below that for all $t \in [0,T]$ the function $\psi(t-\cdot)\star \ol L \colon [0,t] \to (\C^d)^*$ is right-continuous and of bounded variation. Since \eqref{eq:riccati_multdim_kts_tilde_psi_star_L} shows that $\psi(T-\cdot) \star \ol L$ is continuous and of bounded variation, it then follows that for all $t \in [0,T]$ also $g_t \colon [0,t] \to (\C^d)^*$ is right-continuous and of bounded variation.
	
	To derive the expressions \eqref{eq:repres_Y_cov_case_early} and  \eqref{eq:repres_Y_conv_case} for $Y$, observe that  we obtain from the definition \eqref{eq:def_Y_kts} of $Y_0$ together with Remark \ref{rem:affine_volt_proc_conv} that \begin{equation*}
	\begin{split}
	Y_0 
	& = (\psi(T-\cdot) \star \ol L)(T) X_0 + \int_0^T \psi(s) b^0(s) ds + \frac12 \int_0^T \psi(s) A^0(s) \psi(s)^T ds.
	\end{split}
	\end{equation*}
	Inserting this in the definition \eqref{eq:def_Y_kts} of $Y$ yields for all $t \in [0,T]$
	\begin{equation}\label{eq:repres_Y_a02}
	\begin{split}
	Y_t & = (\psi(T-\cdot) \star \ol L)(T) X_0 + \int_0^T \psi(s) b^0(s) ds + \frac12 \int_0^T \psi(s) A^0(s) \psi(s)^T ds \\
	& \quad + \int_0^t \psi(s) \sigma(s,X_s) dW_s - \frac12 \int_0^t \psi(s) a(s,X_s) \psi(s)^T ds \\ 
	& = (\psi(T-\cdot) \star \ol L)(T) X_0 + \int_0^t \psi(s) b(s,X_s) ds + \int_0^t \psi(s) \sigma(s,X_s) dW_s 
	+ \int_t^T \psi(s) b^0(s) ds \\
	& \quad + \frac12 \int_t^T \psi(s) A^0(s) \psi(s)^T ds 
	- \int_0^t \left( \psi(s) B(s) + \frac12 A(s,\psi(s)) \right)X_s ds . 
	\end{split}
	\end{equation}
	Let $Z=\int_0^{\cdot} b(s,X_s) ds + \int_0^{\cdot} \sigma(s,X_s) dW_s$. It then holds for all $t \in [0,T]$ 
	\begin{equation}\label{eq:repres_Y_a03}
	\left(\psi(t-\cdot) \star dZ\right)(t) = \int_0^t \psi(s) b(s,X_s) ds + \int_0^t \psi(s) \sigma(s,X_s) dW_s. 
	\end{equation}
	Since for all $t \in [0,T]$ the function $\psi(t-\cdot)\star \ol L \colon [0,t] \to (\C^d)^*$ is right-continuous and of bounded variation (see part 2 below), \cite[Lemma 2.6]{Jaber2019affine} shows that for all $t \in [0,T]$ 
	\begin{equation}\label{eq:repres_Y_a04}
	\begin{split}
	(\psi(t-\cdot) \star dZ)(t) & = \left( \psi(t-\cdot) \star \ol L \right)(0)X_t - \left( \psi(t-\cdot) \star \ol L \right)(t) X_0 \\
	& \quad + \left(d\left( \psi(t-\cdot) \star \ol L \right)\star X \right)(t).
	\end{split}
	\end{equation}
	We combine \eqref{eq:repres_Y_a02}, \eqref{eq:repres_Y_a03} and \eqref{eq:repres_Y_a04} to obtain for all $t \in [0,T]$
	\begin{equation*}
	\begin{split}
	Y_t & = (\psi(T-\cdot) \star \ol L)(T) X_0 + 
	\psi(t) \ol L(\{0\})X_t - \left( \psi(t-\cdot) \star \ol L \right)(t) X_0 \\
	& \quad + \left(d\left( \psi(t-\cdot) \star \ol L \right)\star X \right)(t) 
	+ \int_t^T \psi(s) b^0(s) ds \\
	& \quad + \frac12 \int_t^T \psi(s) A^0(s) \psi(s)^T ds 
	- \int_0^t \left( \psi(s) B(s) + \frac12 A(s,\psi(s)) \right)X_s ds.
	\end{split}
	\end{equation*}
	This proves \eqref{eq:repres_Y_cov_case_early}. Furthermore, 
	observe that by \eqref{eq:g_t_as_conv} we have for all $t \in [0,T]$ 
	\begin{equation*}
	(\psi(T-\cdot) \star \ol L)(T) X_0 - \left( \psi(t-\cdot) \star \ol L \right)(t)X_0 = -g_t(t)X_0
	\end{equation*}
	and 
	\begin{equation*}
	d\left( \psi(t-\cdot) \star \ol L \right) = dg_t + \mu_t
	\end{equation*}
	where $\mu_t$ denotes the measure associated to $r \mapsto (\psi(T-\cdot) \star \ol L)(T-t+r)$. 
	It therefore follows for all $t \in [0,T]$
	\begin{equation}\label{eq:repres_Y_a06}
	\begin{split}
	Y_t & = -g_t(t) X_0 + 
	\psi(t) \ol L(\{0\})X_t + \left( (dg_t) \star X \right)(t) + \left( \mu_t \star X \right)(t) 
	+ \int_t^T \psi(s) b^0(s) ds \\
	& \quad + \frac12 \int_t^T \psi(s) A^0(s) \psi(s)^T ds 
	- \int_0^t \left( \psi(s) B(s) + \frac12 A(s,\psi(s)) \right)X_s ds.
	\end{split}
	\end{equation}
	Note that by \eqref{eq:riccati_multdim_kts_tilde_psi_star_L} the measure $\mu_t$ for each $t \in [0,T]$ is given by
	\begin{equation*}
 	\mu_t((r,l])
 	=  \int_{r}^{l} f(t-s) + \psi(t-s) B(t-s) + \frac12 A(t-s,\psi(t-s)) ds, \quad r,l \in [0,t], r<l.
	\end{equation*}
	It thus holds that 
	\begin{equation*}
	\begin{split}
	\int_{[0,t]} \mu_t(ds) X_{t-s} 
	& = \int_0^t \left( f(t-s) + \psi(t-s)  B(t-s) + \frac12 A(t-s,\psi(t-s)) \right) X_{t-s} ds \\
	& = \int_0^t \left( f(r) + \psi(r) B(r) + \frac12 A(r,\psi(r)) \right) X_r dr, \quad t \in [0,T].
	\end{split}
	\end{equation*}
	Substituting this in \eqref{eq:repres_Y_a06} yields \eqref{eq:repres_Y_conv_case}.
	
2. It remains to prove that for all $t \in [0,T]$ the function $\psi(t-\cdot)\star \ol L \colon [0,t] \to (\C^d)^*$ is right-continuous and of bounded variation.	
	Note that \eqref{eq:riccati_multdim_kts_tilde_psi_star_L} shows that $\psi(T-\cdot) \star \ol L$ is continuous and of bounded variation. We therefore fix $t \in [0,T)$ for the remainder of the proof. 
	Denote  
	\begin{equation*}
	G\colon [0,T] \to (\C^d)^*, \quad G(s) =  f(s) + \psi(s) B(s)+\frac{1}{2}A(s,\psi(s)), \quad s \in [0,T]. 
	\end{equation*}
	It then holds by the Riccati-Volterra equation \eqref{eq:riccati_multdim_kts} and Fubini's theorem that for all $r \in [0,t]$
	\begin{equation*}
	\begin{split}
	& (\psi(t-\cdot) \star \ol L)(r) 
	= \int_{[0,r]} \psi(t-r+z)\ol L(dz) \\
	& = u \int_{[0,r]} \Delta_{T-t}\ol K(r-z) \ol L(dz) 
	+ \int_{[0,r]} \int_{z-r}^{T-t} G(t+s) \ol K(r-z+s) ds \ol L(dz) \\
	& = u ((\Delta_{T-t} \ol K) \star \ol L)(r) + \int_{[0,r]} \int_{z-r}^{0}  G(t+s) \ol K(r-z+s) ds \ol L(dz) \\
	& \quad + \int_{[0,r]} \int_{0}^{T-t} G(t+s) \ol K(r-z+s) ds \ol L(dz) \\
	& = u ((\Delta_{T-t} \ol K) \star \ol L)(r) + \int_{-r}^0 G(t+s) \int_{[0,s+r]} \ol K(s+r-z) \ol L(dz) ds \\
	& \quad + \int_{0}^{T-t}  G(t+s) \int_{[0,r]} \ol K(r-z+s) \ol L(dz) ds \\
	& = u ((\Delta_{T-t} \ol K) \star \ol L)(r) + \int_{t-r}^{t} G(s) ds + \int_{0}^{T-t}  G(t+s) \left( (\Delta_s \ol K) \star \ol L \right)(r) ds.
	\end{split}
	\end{equation*}
	Note that $[0,t] \ni r \mapsto \int_{t-r}^{t} G(s) ds \in (\C^d)^*$ is continuous and of bounded variation. 
	Due to the assumption $\sup_{r \leq T} \| (\Delta_r \ol K) \star \ol L \|_{TV} < \infty$ we also have that for all $s\in [0,T]$, $(\Delta_{s} \ol K) \star \ol L$ is of bounded variation. 
	For $s \in (0,T]$ it holds by assumption that $\Delta_{s} \ol K$ is continuous on $[0,T]$ and therefore $(\Delta_{s}\ol K) \star \ol L$ is right-continuous on $[0,T)$. 
	Taking the supremum over partitions $0\leq t_1 <\ldots < t_N \leq t$, $N \in \N$, we have 
	\begin{equation*}
	\begin{split}
	& \sup \sum_{j=1}^{N-1} \left\| \int_0^{T-t} G(t+s) \left( (\Delta_s \ol K) \star \ol L \right)(t_{j+1}) ds - \int_0^{T-t} G(t+s) \left( (\Delta_s \ol K) \star \ol L \right)(t_j) ds \right\| \\
	& \leq \sup \int_0^{T-t} \| G(t+s) \| \sum_{j=1}^{N-1} \left\| \left( (\Delta_s \ol K) \star \ol L \right)(t_{j+1}) - \left( (\Delta_s \ol K) \star \ol L \right)(t_j) \right\| ds \\
	& \leq \int_0^{T-t} \| G(t+s) \| ds \, \sup_{r \leq T} \| (\Delta_r \ol K) \star \ol L \|_{TV} .
	\end{split}
	\end{equation*}
	Together with the assumption  $\sup_{r \leq T} \| (\Delta_r \ol K) \star \ol L \|_{TV} < \infty$ and the fact that $G \in L^1\left( [0,T], (\C^d)^* \right)$ this shows that $[0,t] \ni r \mapsto \int_0^{T-t} G(t+s) \left( (\Delta_s \ol K) \star \ol L \right)(r) ds \in (\C^d)^*$ is of bounded variation. 
	For $r \in [0,t]$ and $\varepsilon>0$ such that $r+\varepsilon< T$ it holds 
	\begin{equation}\label{eq:right_cont_G_by_domin_conv}
	\begin{split}
	& \left\| \int_0^{T-t} G(t+s) \left( (\Delta_s \ol K) \star \ol L \right)(r+\varepsilon) ds - \int_0^{T-t} G(t+s) \left( (\Delta_s \ol K) \star \ol L \right)(r) ds \right\| \\
	& \leq \int_0^{T-t} \|G(t+s)\| \, \left\|\left( (\Delta_s \ol K) \star \ol L \right)(r+\varepsilon) - \left( (\Delta_s \ol K) \star \ol L \right)(r) \right\| ds, 
	\end{split}
	\end{equation}
	and, for all $s \in [0,T-t]$,
	\begin{equation*}
	\begin{split}
	& \left\|\left( (\Delta_s \ol K) \star \ol L \right)(r+\varepsilon) - \left( (\Delta_s \ol K) \star \ol L \right)(r) \right\| 
	\leq \sup_{z \leq T} \| (\Delta_z \ol K) \star \ol L \|_{TV} < \infty .
	\end{split}
	\end{equation*}
	By right-continuity of $(\Delta_s\ol K) \star \ol L$ on $[0,T)$ for all $s \in [0,T-t]$ it holds that 
	\begin{equation*}
	\left\|\left( (\Delta_s \ol K) \star \ol L \right)(r+\varepsilon) - \left( (\Delta_s \ol K) \star \ol L \right)(r) \right\| \to 0 \textrm{ as } \varepsilon \downarrow 0.
	\end{equation*}
	The assumption that $G \in L^1\left( [0,T], (\C^d)^* \right)$, the dominated convergence theorem and \eqref{eq:right_cont_G_by_domin_conv} hence imply that 
	\begin{equation*}
	\left\| \int_0^{T-t} G(t+s) \left( (\Delta_s \ol K) \star \ol L \right)(r+\varepsilon) ds - \int_0^{T-t} G(t+s) \left( (\Delta_s \ol K) \star \ol L \right)(r) ds \right\| \to 0 \textrm{ as } \varepsilon \downarrow 0.
	\end{equation*}
	Altogether it follows that  $\psi(t-\cdot)\star \ol L\colon [0,t] \to (\C^d)^*$ is right-continuous and of bounded variation.
\end{proof}

\section{The inhomogeneous Volterra-Heston model}\label{sec:inhom_volt_heston_model}

To introduce the inhomogeneous Volterra-Heston model,  
let $d=m=2$, $\rho\in [-1,1]$ 
and let $\kappa, \theta,\bar\sigma,\eta \colon [0,T] \to [0,\infty)$ be continuous such that $\bar\sigma$ is strictly positive. 
Consider for initial values $S_0\in (0,\infty)$, $V_0 \in [0,\infty)$ the bivariate process
\begin{align}\label{eq:price_heston}
S_t&=S_0+\int_0^tS_r\eta(r)\sqrt{V_r}\left(\sqrt{1-\rho^2}dW^{(1)}_r+\rho dW^{(2)}_r\right),\\
V_t&=V_0+\int_0^tk(t-r)\kappa(r)(\theta(r)-V_r)dr+\int_0^tk(t-r)\bar\sigma(r)\sqrt{V_r}dW^{(2)}_r, \quad t \in [0,T], \label{eq:var_heston}
\end{align}
where $W=(W^{(1)},W^{(2)})^T$ is a two-dimensional Brownian motion. 
For $k$ we will later choose the fractional kernel $k(t)=\frac{t^{\alpha-1}}{\Gamma(\alpha)}$, $t\in (0,T]$, with $\alpha\in (\frac{1}{2},1)$. 
This is a particular example (cf. \cite[Example 3.7 and Example~6.2]{Jaber2019affine}) within the following class of kernels: 
\begin{assumption}\label{assumption:special_kernel}
	Let $k\in L^2([0,T], [0,\infty))$ such that $k$ is not identically zero, nonincreasing and continuous on $(0,T]$ and that assumption \ref{assumption:kernel_order} (see Section~\ref{sec:existence_of_soln_of_convolution_SDE}) is satisfied. Moreover, assume that for every $h\in [0,1]$  
	the kernel $[0,T]\ni t\mapsto k((t+h)\wedge T)\in [0,\infty)$ admits a nonnegative resolvent $L^h$ of the first kind on $[0,T]$ which satisfies for every $t\in [0,T]$ that $[0,T]\ni s\mapsto L^h([s,(s+t)\wedge T])\in [0,\infty)$ is nonincreasing. 
	Denote by $L$ the resolvent of the first kind of $k$. 
\end{assumption} 

We call \eqref{eq:price_heston} and \eqref{eq:var_heston} the inhomogeneous Volterra-Heston model. 
Note that if $\eta(t)=1$ for all $t \in [0,T]$ and $\kappa, \theta,\bar\sigma$ are constant, we recover the Volterra-Heston model proposed in \cite[Section 7]{Jaber2019affine}. 
If furthermore $k$ is a fractional kernel then we recover the rough Heston model of \cite{Euch2019characteristicFctRoughHeston}, \cite{Euch2018perfectHedging} and \cite{Euch2019rougheningHeston}. 
In the case $k(t)=1$ for all $t \in [0,T]$, this is the common Heston model \cite{Heston1993closedformSoln}.

In this section we obtain results analogous to \cite[Theorem 7.1]{Jaber2019affine} also for the inhomogeneous Volterra-Heston model. 
First of all, we establish existence of a continuous weak solution for \eqref{eq:price_heston} and \eqref{eq:var_heston}, and that the price process $S$ is a martingale (see Proposition~\ref{propo:existence_S_V}). 
Next, we show that it is possible to introduce the time-dependent continuous functions  $\kappa,\theta,\bar\sigma$ and $\eta$ into the proof of \cite[Lemma 7.4]{Jaber2019affine} to establish that under the same conditions as in \cite[Lemma 7.4]{Jaber2019affine}, the Riccati-Volterra equation \eqref{eq:riccati_multdim_kts} has a unique solution $\psi$ in $L^2\left( [0,T], (\C^2)^* \right)$ (see Proposition~\ref{propo:ex_unique_soln_riccati_volt_eq}). Moreover, we prove that under more restrictive assumptions $\psi$ is bounded. 
We then establish in  Proposition~\ref{propo:exp_Y_is_true_martingale_heston} that $e^Y$ with $Y$ defined by~\eqref{eq:def_Y_kts} is a true martingale, and thus the conditional Fourier-Laplace formula~\eqref{eq:conditional_Fourier_Laplace_formula} of Theorem~\ref{thm:multdim_volt_kts} holds true. 
The proof is facilitated by Theorem~\ref{thm:represent_for_Y}, which replaces \cite[Theorem 4.5]{Jaber2019affine}. 
Subsequently, we argue in Corollary~\ref{cor:uniqueness_S_V} that any continuous weak solution for \eqref{eq:price_heston} and \eqref{eq:var_heston} is unique.  
We finally apply the previous results to the inhomogeneous Volterra-Heston model with fractional kernel.

\begin{propo}\label{propo:existence_S_V}  
Let \ref{assumption:special_kernel} be satisfied. Then, equation \eqref{eq:var_heston} possesses a $[0,\infty)$-valued continuous weak solution $V$. 
Moreover, the paths of $V$ are H\"older continuous of any order strictly smaller than $\gamma/2$ (where $\gamma \in (0,2]$ comes from \ref{assumption:kernel_order}). 
Furthermore, 
\begin{equation}\label{eq:formula_for_S_inhom_Heston}
S_t=S_0e^{\int_0^t\eta(r)\sqrt{V_r}\left(\sqrt{1-\rho^2}dW^{(1)}_r+\rho dW^{(2)}_r\right)-\int_0^t \frac{\eta^2(r)V_r}{2}dr}, \quad t\in [0,T],
\end{equation}
is a solution of \eqref{eq:price_heston}, and $S$ is a martingale. 
\end{propo}

\begin{proof} 
	For $t \in [0,T]$, $x \in \R$ let $\hat b(t,x)=\kappa(t)(\theta(t)-x)$ and $\hat\sigma(t,x)=\bar\sigma(t)\sqrt{x}1_{[0,\infty)}(x)$, and  
	consider the stochastic Volterra equation  
	\begin{equation}\label{eq:equ_for_V_inhom_heston}
	V_t = V_0+\int_0^tk(t-r)\hat b(r,V_r)dr+\int_0^tk(t-r)\hat\sigma(r,V_r)dW^{(2)}_r, \quad t \in [0,T].
	\end{equation}
	Note that due to continuity of $\kappa$, $\theta$ and $\bar\sigma$ it holds that $\hat b$ and $\hat \sigma$ are in $C([0,T]\times\R, \R)$ and that \ref{assumption:linear_growth} is satisfied. 
	Furthermore, for all $t \in [0,T]$ we have that $\hat b(t,0) = \kappa(t)\theta(t)\geq 0$ and $\hat\sigma(t,0)=0$. Therefore, taking also into account assumption \ref{assumption:special_kernel} and the fact that $V_0 \in [0,\infty)$, it follows from Proposition~\ref{propo:existence_weak_soln_sde_conv_nonneg} that equation \eqref{eq:equ_for_V_inhom_heston} admits a $[0,\infty)$-valued continuous weak solution $V$. 
	This is also a solution of equation \eqref{eq:var_heston}.  
	Lemma~\ref{lem:moments_of_soln_bdd} and (A2) imply that for any $p\geq 1$, $\sup_{t \in [0,T]} E[ \lvert \hat\sigma(t,V_t) \rvert^p + \lvert \hat b(t,V_t) \rvert^p ]$ is finite. Thus, \cite[Lemma 2.4]{Jaber2019affine} applies and yields that the paths of $V$ are Hölder continuous of any order strictly smaller than $\gamma/2$.

	Finally, note that $S$ given by \eqref{eq:formula_for_S_inhom_Heston} clearly solves \eqref{eq:price_heston}. 
	Since $\eta$ is continuous and hence bounded on $[0,T]$, Lemma~\ref{lemma:e_U_is_martingale} proves that $S$ is a martingale. 
\end{proof}

For the remainder of this section let $V$ be a $[0,\infty)$-valued continuous weak solution of \eqref{eq:var_heston} and let $S$ be the solution of \eqref{eq:price_heston} given by \eqref{eq:formula_for_S_inhom_Heston}.
 
Note that It\^o's formula shows that for all $t \in [0,T]$ the log-price is given by
\begin{equation}\label{eq:log_S_t}
\log(S_t)=\log(S_0)+\int_0^t\eta(r)\sqrt{V_r}\left(\sqrt{1-\rho^2}dW^{(1)}_r+\rho dW^{(2)}_r\right)-\int_0^t\frac{\eta^2(r)V_r}{2}dr.
\end{equation}
Hence the pair $X=(\log(S),V)^T$ satisfies
\begin{equation}
\begin{split}
\begin{pmatrix}
X^{(1)}_t \\ X^{(2)}_t
\end{pmatrix}
&=\begin{pmatrix}
X^{(1)}_0 \\ X^{(2)}_0
\end{pmatrix}
+\int_0^t
\begin{pmatrix}
1 & 0 \\ 0 & k(t-r)
\end{pmatrix}
\left[
\begin{pmatrix}
0 \\ \kappa(r) \theta(r)
\end{pmatrix}
+
\begin{pmatrix}
0 \\ 0
\end{pmatrix}
X^{(1)}_r 
+
\begin{pmatrix}
-\frac{\eta^2(r)}{2} \\ -\kappa(r)
\end{pmatrix}
X^{(2)}_r
\right]
dr\\
&\quad
+\int_0^t
\begin{pmatrix}
1 & 0 \\ 0 & k(t-r)
\end{pmatrix}
\begin{pmatrix}
\eta(r)\sqrt{1-\rho^2} & \eta(r)\rho \\
0 & \bar\sigma(r)
\end{pmatrix}
\sqrt{X^{(2)}_r} 
dW_r, \quad t \in [0,T].
\end{split}
\end{equation}
For $t \in [0,T]$, $x \in \R^2$ let
\begin{equation}\label{eq:coeffs_heston}
\begin{split}
&\ol K(t)=\begin{pmatrix}
1 & 0 \\ 0 & k(t)
\end{pmatrix}, 
\quad \ol L=\begin{pmatrix}
\delta_0 & 0 \\ 0 & L
\end{pmatrix},
\\
& b^0(t)=\begin{pmatrix}
0 \\ \kappa(t) \theta(t)
\end{pmatrix},
\quad
B(t)=
\begin{pmatrix}
b^1(t) & b^2(t)
\end{pmatrix}
=
\begin{pmatrix}
0 & -\frac{\eta^2(t)}{2} \\ 
0 & -\kappa(t)
\end{pmatrix},
\quad 
b(t,x) = b^0(t)+B(t)x ,\\
&\sigma(t,x)=
\begin{pmatrix}
\eta(t)\sqrt{1-\rho^2} & \eta(t)\rho \\
0 & \bar\sigma(t)
\end{pmatrix} 
\sqrt{x_2}1_{[0,\infty)}(x_2),
\quad 
A^0(t)=A^1(t)=0,\\
&A^2(t)=\begin{pmatrix}
\eta^2(t) & \rho \eta(t) \bar\sigma(t) \\ 
\rho \eta(t) \bar\sigma(t) & \bar\sigma^2(t)
\end{pmatrix},
\quad
a(t,x)=A^0(t)+A^1(t)x_1+A^2(t)x_2,
\end{split}
\end{equation}
and observe that \ref{assumption:affine} with $E=\R\times[0,\infty)$ and \ref{assumption:linear_growth} are  satisfied and that $\sigma \in C\left([0,T]\times\R^2,\R^{2\times 2}\right)$.
Then 
\begin{equation}
X_t=X_0+\int_0^t \ol K(t-s)b(s,X_s)ds+\int_0^t \ol K(t-s)\sigma(s,X_s)dW_s, \quad t \in [0,T].
\end{equation}
With \eqref{eq:coeffs_heston} we thus obtain that for the inhomogeneous Volterra-Heston model the Riccati-Volterra equation \eqref{eq:riccati_multdim_kts} for any  $u \in (\C^2)^*$ and $f \in L^1\left( [0,T], (\C^2)^* \right)$ reads 
\begin{align}
\label{eq:riccati_heston_gen_a}
\psi_1(t) & = u_1 + \int_t^T f_1(s)ds,\\
\psi_2(t) & = u_2 k(T-t) + \int_t^T \Big( f_2(s) +  \frac{\eta^2(s)}{2} (\psi^2_1(s)-\psi_1(s)) - \kappa(s) \psi_2(s) + \frac{\bar\sigma^2(s)}{2} \psi_2^2(s) \nonumber \\
&\quad + \rho \bar\sigma(s)  \eta(s) \psi_1(s) \psi_2(s) \Big) k(s-t)ds, \quad t \in [0,T]. \label{eq:riccati_heston_gen_b}
\end{align}
Denote $\widetilde{\psi}(r)=\psi(T-r)$, $r \in [0,T]$, and observe that $\psi$ solves \eqref{eq:riccati_heston_gen_a} and  \eqref{eq:riccati_heston_gen_b} if and only if $\widetilde \psi$ solves 
\begin{equation}
\begin{split}\label{eq:riccati_heston_gen_tilde}
\widetilde\psi_1(t) & = u_1 + \int_0^t f_1(T-s)ds,\\
\widetilde\psi_2(t) & = u_2 k(t) + \int_0^t \Big( f_2(T-s) +  \frac{\eta^2(T-s)}{2} (\widetilde\psi^2_1(s)-\widetilde\psi_1(s)) \\
& \quad + \left( \rho \bar\sigma(T-s) \eta(T-s) \widetilde{\psi}_1(s) - \kappa(T-s) \right) \widetilde \psi_2(s) + \frac{\bar\sigma^2(T-s)}{2} \widetilde \psi_2^2(s) \Big) k(t-s)ds, \\
& \quad t \in [0,T].
\end{split}
\end{equation}

\begin{propo}\label{propo:ex_unique_soln_riccati_volt_eq} 
	Let \ref{assumption:special_kernel} be satisfied. 
	Suppose that $u \in (\C^2)^*$ and $f \in L^1\left( [0,T], (\C^2)^* \right)$ such that $\Re \psi_1 \in [0,1]$, $\Re u_2\leq 0$ and $\Re f_2 \leq 0$, where $\psi_1 = u_1 + \int_{\cdot}^T f_1(s)ds$. 
	Then the Riccati-Volterra equation \eqref{eq:riccati_heston_gen_b} has a unique solution $\psi_2 \in L^2([0,T],\C)$. Moreover, it satisfies $\Re \psi_2 \leq 0$. 
	If, in addition, $u_2=0$, $L \star \Im f_1(T-\cdot) \in L^2\left( [0,T], \R \right)$ and $f_2 \in L^2\left( [0,T], \C \right)$, then $\psi_2$ is bounded.
\end{propo} 

\begin{proof}
	Using that $\kappa, \bar\sigma, \eta$ are continuous and hence bounded on $[0,T]$, and that $\bar\sigma>0$, the same reasoning as in the proof of \cite[Lemma 7.4]{Jaber2019affine} applies to $\widetilde\psi$. 
	In particular, the boundedness guarantees that all integrability conditions remain satisfied and that \cite[(B.2)]{Jaber2019affine} continues to hold.
 
	We obtain from \cite[Theorem B.1]{Jaber2019affine} the existence of a unique non-continuable solution $(\widetilde\psi_2,T_{\max})$ of  \eqref{eq:riccati_heston_gen_tilde} (see \cite[Appendix B]{Jaber2019affine} for a definition), which by \cite[Theorem C.2]{Jaber2019affine} satisfies $\Re \widetilde\psi_2 \leq 0$.  
	Note that in our setting, the equations for $h$ and $l$ in the proof of \cite[Lemma 7.4]{Jaber2019affine} become 
	\begin{equation*}
	\begin{split}
	h & = |\Im u_2 | k + |\rho \bar\sigma^{-1}(T-\cdot) \eta(T-\cdot) \Im u_1 | + k \star \bigg( \Big| \rho \bar\sigma^{-1}(T-\cdot) \eta(T-\cdot) (L \star \Im f_1(T-\cdot)) \\
	& \quad + \Im f_2(T-\cdot) 
	+ \frac{\Im \widetilde\psi_1}{2} \Big( 2 \eta^2(T-\cdot) (1-\rho^2) \Re\widetilde{\psi}_1 - \eta^2(T-\cdot) \\
	& \quad + 2\rho\kappa(T-\cdot)\bar\sigma^{-1}(T-\cdot) \eta(T-\cdot) \Big) \Big| 
	+ (\rho\bar\sigma(T-\cdot) \eta(T-\cdot) \Re\widetilde\psi_1 - \kappa(T-\cdot)) h \bigg),\\
	l & = \Re u_2 k + k \star \bigg( \Re f_2(T-\cdot) 
	+ \frac{\eta^2(T-\cdot)}{2} \left( (\Re \widetilde\psi_1)^2 - \Re \widetilde\psi_1 - (\Im \widetilde\psi_1)^2 \right) \\
	& \quad - | \rho \bar\sigma(T-\cdot) \eta(T-\cdot) \Im \widetilde\psi_1 | (h + | \rho \bar\sigma^{-1}(T-\cdot) \eta(T-\cdot) \Im\widetilde\psi_1 |) \\
	& \quad - \frac{\bar\sigma^2(T-\cdot)}{2} (h + | \rho \bar\sigma^{-1}(T-\cdot) \eta(T-\cdot) \Im\widetilde\psi_1 |)^2 \\
	& \quad + (\rho\bar\sigma(T-\cdot) \eta(T-\cdot) \Re\widetilde\psi_1 - \kappa(T-\cdot)) l \bigg) .
	\end{split}
	\end{equation*}
	By \cite[Corollary B.3]{Jaber2019affine} there exist unique global solutions $h,l \in L^2([0,T],\R)$.  
	As in the proof of \cite[Lemma 7.4]{Jaber2019affine} one shows that $l \leq \Re \widetilde\psi_2 \leq 0$ and $| \Im \widetilde\psi_2 | \leq h + | \rho \bar\sigma^{-1}(T-\cdot) \eta(T-\cdot) \Im\widetilde\psi_1 |$. 
	From this it follows by the same arguments as in the proof of \cite[Lemma 7.4]{Jaber2019affine} that there exists a unique global solution $\widetilde \psi_2 \in L^2([0,T],\C)$ of \eqref{eq:riccati_heston_gen_tilde} which satisfies $\Re \widetilde\psi_2 \leq 0$.  
	
	It remains to prove that $\widetilde \psi_2$ is bounded under the additional assumptions that $u_2=0$, $L \star \Im f_1(T-\cdot) \in L^2\left( [0,T], \R \right)$ and $f_2 \in L^2\left( [0,T], \C \right)$. 
	Since $\kappa,\bar\sigma,\bar\sigma^{-1},\eta$ and $\widetilde\psi_1 = u_1 + \int_0^{\cdot} f_1(T-s)ds$ are bounded on $[0,T]$, the Cauchy-Schwarz inequality implies first that $h$ is bounded on $[0,T]$ and subsequently the same for $l$. 
	This yields boundedness of $\Re \widetilde\psi_2$ and $\Im \widetilde\psi_2$.
\end{proof}

\begin{propo}\label{propo:exp_Y_is_true_martingale_heston}
	Let \ref{assumption:special_kernel} be satisfied. 
	Suppose that $u \in (\C^2)^*$ and $f \in L^1\left( [0,T], (\C^2)^* \right)$ such that $\Re \psi_1 \in [0,1]$, $\Re u_2\leq 0$ and $\Re f_2 \leq 0$, where $\psi_1 = u_1 + \int_{\cdot}^T f_1(s)ds$. 
	Let $\psi_2 \in L^2([0,T],\C)$ denote the unique solution of  \eqref{eq:riccati_heston_gen_b} (cf. Proposition \ref{propo:ex_unique_soln_riccati_volt_eq}). 
	Then $exp(Y)$ with $Y$ defined by \eqref{eq:def_Y_kts} 
	is a true martingale. 
	In particular, formula \eqref{eq:conditional_Fourier_Laplace_formula} holds.
\end{propo}

\begin{proof} 	
	The proof is similar to the one of \cite[Theorem 7.1]{Jaber2019affine}. 
	For $t \in [0,T]$ 
	let $g_t$ be defined by \eqref{eq:def_pi_h} and observe that the first component $g_{t,1}$ of $g_t$ equals $0$ in the current setting. As in the proof of \cite[Theorem 6.1]{Jaber2019affine} it follows from \cite[(3.9) and (3.10)]{Jaber2019affine} that $\sup_{r\leq T}\| (\Delta_r\ol K)\star \ol L \|_{TV}<\infty$. 
	Hence, we can apply Theorem~\ref{thm:represent_for_Y}(ii)  to obtain that 
	\begin{equation*}
	\begin{split}
	Y_t & = -g_{t,2}(t) V_0 + \psi_1(t) \log(S_t)
	+ \psi_2(t) L(\{0\})V_t 
	+ \int_{[0,t]} (d\Re g_{t,2}(s)) V_{t-s} \\
	& \quad + \int_0^t f_1(r)\log(S_r)dr 
	+ \int_0^t f_2(r)V_r dr 
	+  \int_t^T \kappa(s) \theta(s) \psi_2(s)ds, \quad t \in [0,T].
	\end{split}
	\end{equation*}
	Consider the real part
	\begin{equation}\label{eq:Re_Y_in_propo}
	\begin{split}
	\Re Y_t & = -\Re g_{t,2}(t) V_0 +  \Re\psi_1(t) \log(S_t)
	+ \Re \psi_2(t) L(\{0\})V_t 
	+ \int_{[0,t]} (d\Re g_{t,2}(s)) V_{t-s} \\
	& \quad + \int_0^t \Re f_1(r)\log(S_r)dr 
	+ \int_0^t \Re f_2(r)V_r dr 
	+  \int_t^T \kappa(s) \theta(s) \Re\psi_2(s)ds, \quad t \in [0,T],
	\end{split}
	\end{equation}
	and recall that for $t \in [0,T]$ 
	\begin{equation*}
	\begin{split}
	\Re g_{t,2}(r) & = -\int_{(r,r+T-t]} \Re \psi_2(t+s-r)L(ds)
	= \int_{(t,T]} - \Re \psi_2(s) L(r-t+ds)
	, \quad r \in [0,t].
	\end{split}
	\end{equation*}
	Due to $\Re\psi_2\leq 0$ and nonnegativity of $L$ we have that $\Re g_{t,2}$ is nonnegative for all $t \in [0,T]$. 
	This yields that $-\Re g_{t,2}(t)V_0\leq 0$ for all $t \in [0,T]$.  
	Furthermore, observe that the facts that $\Re\psi_2\leq 0$ and that $L$ is nonincreasing (in the sense of assumption \ref{assumption:special_kernel}) imply that for all $t \in [0,T]$ and $r_1<r_2$ in $[0,t]$ 
	\begin{equation*}
	\begin{split}
	\Re g_{t,2}(r_2) - \Re g_{t,2}(r_1) 
	& = \int_{(t,T]} - \Re \psi_2(s) \left(L(r_2-t+ds) - L(r_1-t+ds) \right) \leq 0.
	\end{split}
	\end{equation*}
	Hence, $\Re g_{t,2}$ is nonincreasing for all $t \in [0,T]$.	
	Since $V$ is $[0,\infty)$-valued, it follows that 
	$\int_{[0,t]} (d\Re g_{t,2}(s)) V_{t-s} \leq 0$ for all $t \in [0,T]$. 
	Moreover, the facts that $V$ is $[0,\infty)$-valued, that $L$ is nonnegative, that $\Re \psi_2 \leq 0$ and that $\Re f_2\leq 0$ imply that $\Re \psi_2(t)L(\{0\})V_t \leq 0$ and $\int_0^t \Re f_2(r)V_rdr\leq 0$ for all $t \in [0,T]$. 
	We also have by $\Re \psi_2 \leq 0$ and $\kappa,\theta\geq 0$ that $\int_t^T \kappa(s)\theta(s) \Re \psi_2(s)ds \leq 0$ for all $t \in [0,T]$. 
	Altogether, the preceding discussion shows that it follows from \eqref{eq:Re_Y_in_propo} that 
	\begin{equation*}
	\begin{split}
	\Re Y_t & \leq \Re\psi_1(t) \log(S_t)
	+ \int_0^t \Re f_1(r)\log(S_r)dr, \quad t \in [0,T].
	\end{split}
	\end{equation*}
	Integration by parts then yields that 
	\begin{equation}\label{eq:Y_heston_estimate_int_by_parts}
	\begin{split}
	\Re Y_t & \leq \Re\psi_1(0) \log(S_0)
	+ \int_0^t \Re \psi_1(r)d\log(S_r), \quad t \in [0,T].
	\end{split}
	\end{equation} 
	Define 
	\begin{equation*}
	\begin{split}
	U_t & = \int_0^t \Re\psi_1(s) \eta(s) \sqrt{V_s} \left( \sqrt{1-\rho^2} dW_s^{(1)} + \rho dW_s^{(2)} \right), \quad t \in [0,T],
	\end{split}
	\end{equation*} 
	and note that since $\Re \psi_1 \in [0,1]$ and $V$ is $[0,\infty)$-valued, it follows from \eqref{eq:Y_heston_estimate_int_by_parts} and \eqref{eq:log_S_t} that 
	\begin{equation*}
	\begin{split}
	\Re Y_t & \leq \Re\psi_1(0) \log(S_0)
	+ U_t - \frac12 \langle U \rangle_t, \quad t \in [0,T].
	\end{split}
	\end{equation*} 
	We obtain that 
	\begin{equation*}
	\| \exp(Y_t) \| = \exp(\Re Y_t) \leq S_0^{\Re \psi_1(0)} \exp\left(U_t - \frac12 \langle U\rangle_t\right), \quad t\in[0,T].
	\end{equation*}
	Since $\eta$ and $\Re \psi_1$ are bounded, Lemma~\ref{lemma:e_U_is_martingale} yields that $\exp\left(U - \frac12 \langle U\rangle\right)$ is a martingale, and therefore also $\exp(Y)$ is a true martingale. 
	It follows from Theorem~\ref{thm:multdim_volt_kts} that \eqref{eq:conditional_Fourier_Laplace_formula} holds true.
\end{proof}

\begin{corollary}\label{cor:uniqueness_S_V}
	Let \ref{assumption:special_kernel} be satisfied. Then, any two $[0,\infty)$-valued continuous weak solutions of \eqref{eq:var_heston} have the same law.
\end{corollary}

\begin{proof}
The reasoning is similar to the end of the proof of \cite[Theorem 6.1]{Jaber2019affine}. 
As before, suppose that $V$ is a $[0,\infty)$-valued continuous weak solution of \eqref{eq:var_heston}. 
Let $n \in \N$, $0\leq t_1 < \ldots < t_n \leq T$ and $\lambda_j \in [0,\infty)$, $j \in \{1,\ldots,n\}$. 
Then, $E\left[\exp\left(-\sum_{j=1}^n \lambda_j V_{t_j} \right)\right]$ can be approximated by a sequence  $E\left[\exp\left(\int_0^T f_N(s)V_s ds \right)\right]$, $N \in \N$, where $f_N \in C([0,T],(-\infty,0])$, $N\in \N$. 
For each $N \in \N$ consider the Riccati-Volterra equation \eqref{eq:riccati_multdim_kts} with $u_1=0=u_2$, $f_1\equiv 0$ and $f_2=f_N$. 
By Proposition~\ref{propo:ex_unique_soln_riccati_volt_eq} and Proposition~\ref{propo:exp_Y_is_true_martingale_heston} we have that for all $N \in \N$  
there exists a unique solution $\psi_{2,N} \in L^2([0,T],\C)$ of \eqref{eq:riccati_heston_gen_b}, and  $\left(\exp(Y_{t,N})\right)_{t \in [0,T]}$ with $Y_N=(Y_{t,N})_{t \in [0,T]}$ defined by \eqref{eq:def_Y_kts} is a true martingale. 
It follows from \eqref{eq:time_zero} that for all $N \in \N$
\begin{equation}\label{eq:Laplace_trafo_not_dep_on_V}
\begin{split}
& E\left[ \exp\left( \int_0^T f_{N}(s)V_s ds \right) \right]\\
& = \exp\left( \int_0^T \kappa(s) \theta(s)\psi_{2,N}(s) ds + V_0 \int_0^T f_{N}(s) - \kappa(s) \psi_{2,N}(s) + \frac12 \bar\sigma^2(s) \psi_{2,N}^2(s) ds \right).
\end{split}
\end{equation}
Suppose that $\wh V$ is another $[0,\infty)$-valued  continuous weak solution of \eqref{eq:volterra_d_volt_inhom}. 
We then obtain from \eqref{eq:Laplace_trafo_not_dep_on_V} that $E\left[ \exp\left( \int_0^T f_{N}(s)V_s ds \right) \right]=E\left[ \exp\left( \int_0^T f_{N}(s)\wh V_s ds \right) \right]$, $N\in\N$. 
Therefore, $(V_{t_1},\ldots,V_{t_n})$ and $(\wh V_{t_1},\ldots,\wh V_{t_n})$ possess the same Laplace transform. It follows that the law of $V$ and the law of $\wh V$ are equal.
\end{proof}

The previous results enable us to obtain that the Fourier-Laplace functional (at time zero) in the inhomogeneous Volterra-Heston model with fractional kernel is given in terms of the solution of a fractional Riccati equation. 
A fractional kernel is of particular interest because it is commonly used to model roughness, for example in the rough Heston model (see e.g. \cite{Jaisson2016roughHawkes},  \cite{Euch2019characteristicFctRoughHeston}, \cite{Euch2018perfectHedging} and  \cite{Euch2019rougheningHeston}). 
The rough Heston model constitutes a special case of our inhomogeneous Volterra-Heston model by setting $\eta(t)=1$ for all $t \in [0,T]$, taking $\kappa, \theta,\bar\sigma$ to be constant, and choosing a fractional kernel. 
The following result thus generalizes results in \cite{Euch2019characteristicFctRoughHeston} and \cite[Example 7.2]{Jaber2019affine} to include a time-dependent volatility coefficient $\eta$ and time-dependent $\kappa$, $\theta$ and $\bar\sigma$.
\begin{corollary}
	Let $\alpha\in (\frac{1}{2},1)$ and assume that $k(t)=\frac{t^{\alpha-1}}{\Gamma(\alpha)}$, $t\in (0,T]$.  
	Suppose that $u \in (\C^2)^*$ and $f \in L^1\left( [0,T], (\C^2)^* \right)$ such that $\Re \psi_1 \in [0,1]$, $\Re u_2\leq 0$ and $\Re f_2 \leq 0$, where $\psi_1 = u_1 + \int_{\cdot}^T f_1(s)ds$. 	
	Then there exists a unique solution $\wh\psi \in L^2([0,T],\C)$ of  
	\begin{equation}\label{eq:riccati_heston_final}
	\begin{split}
	(D^{\alpha}\wh \psi)(T-t) & = f_2(t) + \frac{\eta^2(t)}{2}\left( u_1^2 - u_1 + 2u_1 \int_t^{T} f_1(s)ds + \left( \int_t^{T} f_1(s)ds \right)^2 \right) 
	\\
	& \quad +\left(\rho\bar\sigma(t) \left(u_1 + \int_t^{T} f_1(s)ds \right) \eta(t) - \kappa(t) \right)\wh \psi(T-t)
	+\frac{\bar\sigma^2(t)}{2}\wh \psi^2(T-t), \\
	& \quad t \in [0,T],\\
	(I^{1-\alpha}\wh \psi)(0)&=u_2,
	\end{split}
	\end{equation}
	and with $\phi(T)= \int_0^T \kappa(s) \theta(s) \wh\psi(T-s)ds$ it holds 
	\begin{equation}\label{eq:time_zero_frac_heston}
	\begin{split}
	& E\left[ \exp\left( u_1 \log(S_T) + \int_0^T f_1(r) \log(S_r) dr + u_2 V_T  + \int_0^T f_2(r) V_r dr \right) \right] \\
	& = 
	\exp\left( \phi(T) + \left(u_1 + \int_0^T f_1(s)ds \right)\log(S_0) +
	(I^{1-\alpha}\wh \psi)(T)V_0 \right).
	\end{split}
	\end{equation}
\end{corollary}

\begin{proof}
	Note that $k$ satisfies assumption \ref{assumption:special_kernel} (cf. \cite[Example 3.7 and Example~6.2]{Jaber2019affine}). 
	By Proposition~\ref{propo:ex_unique_soln_riccati_volt_eq} there exists a unique solution $\psi_2 \in L^2([0,T],\C)$ of the Riccati-Volterra equation \eqref{eq:riccati_heston_gen_b}. 
	Then \eqref{eq:equi_riccati_frac} shows that $\psi_2(T-\cdot)$ solves \eqref{eq:riccati_heston_final}. 
	From Proposition~\ref{propo:exp_Y_is_true_martingale_heston} we obtain that the process $e^Y$ defined by \eqref{eq:def_Y_kts} is a true martingale. 
	Therefore, \eqref{eq:time_zero_frac} holds true, which in the current setting equals \eqref{eq:time_zero_frac_heston}.
\end{proof}

\appendix

\section{Auxiliary results}

The following lemma provides an inequality that is used in the proof of Lemma~\ref{lem:moments_of_soln_bdd}.
\begin{lemma}\label{lem:fg_estimate_1}
	Let $p \in (2,\infty)$, $d \in \N$, $s \in \R$, $t \in [s,\infty)$, $f \in L^2([s,t],\R^{d\times d})$ and $g \in L^p([s,t],\R^d)$. Then:
	\begin{equation*}
	\begin{split}
	\left\| \int_s^t f(u)g(u) du \right\|^p 
	& \leq d^{\frac{3p}{2}} (t-s)^{\frac{p}{2}} \left( \int_s^t \|f(u)\|^2 \|g(u)\|^p du \right) \left( \int_s^t \|f(u)\|^2 du \right)^{\frac{p}{2}-1} .
	\end{split}
	\end{equation*}
\end{lemma}

\begin{proof}
	It holds by Jensen's inequality that
	\begin{equation*}
	\begin{split}
	\left\| \int_s^t f(u)g(u) du \right\|^p 
	& = \left( \sum_{i=1}^d \left| \sum_{j=1}^d \int_s^t f_{ij}(u) g_j(u) du \right|^2 \right)^{\frac{p}{2}} 
	\leq d^{\frac{p}{2}-1} \sum_{i=1}^d \left| \sum_{j=1}^d \int_s^t f_{ij}(u) g_j(u) du \right|^p \\
	& \leq d^{\frac{3p}{2}-2} \sum_{i=1}^d  \sum_{j=1}^d \left| \int_s^t  f_{ij}(u) g_j(u) du \right|^p .
	\end{split}
	\end{equation*} 
	Observe then that the Cauchy-Schwarz inequality and Jensen's inequality yield for all $i,j \in \{1,\ldots,d\}$ that
	\begin{equation*}
	\begin{split}
	& \left| \int_s^t f_{ij}(u)g_j(u) du \right|^p 
	\leq \left( \int_s^t | f_{ij}(u)|^{\frac{2}{p}} |g_j(u)| |f_{ij}(u)|^{1-\frac{2}{p}} du \right)^{p} \\
	& \leq \left( \int_s^t |f_{ij}(u)|^{\frac{4}{p}} |g_j(u)|^2 du  \right)^{\frac{p}{2}} \left( \int_s^t |f_{ij}(u)|^{2-\frac{4}{p}} du \right)^{\frac{p}{2}} \\
	& \leq (t-s)^{\frac{p}{2}-1} \left( \int_s^t |f_{ij}(u)|^2 |g_j(u)|^p du \right)  \left( \left( \int_s^t |f_{ij}(u)|^{2-\frac{4}{p}} du \right)^{\frac{p}{p-2}} \right)^{\frac{p-2}{2}} \\
	& \leq (t-s)^{\frac{p}{2}-1} \left( \int_s^t |f_{ij}(u)|^2 |g_j(u)|^p du \right) \left( (t-s)^{\frac{p}{p-2}-1} \int_s^t |f_{ij}(u)|^{\left(2-\frac{4}{p}\right)\frac{p}{p-2}} du \right)^{\frac{p-2}{2}} \\
	& = (t-s)^{\frac{p}{2}} \left( \int_s^t |f_{ij}(u)|^2 |g_j(u)|^p du \right) \left( \int_s^t |f_{ij}(u)|^2 du \right)^{\frac{p}{2}-1} .
	\end{split}
	\end{equation*}
\end{proof}

The following lemma shows that \cite[Lemma 7.3]{Jaber2019affine} also holds true in the inhomogeneous Volterra-Heston model.
\begin{lemma}\label{lemma:e_U_is_martingale}
	Suppose the setting of the inhomogeneous Volterra-Heston model of Section~\ref{sec:inhom_volt_heston_model} and that $V$ is a $[0,\infty)$-valued continuous weak solution of \eqref{eq:var_heston}. Let $g\in L^{\infty}([0,T],\R)$ and define 
	\begin{equation}
	U_t = \int_0^t g(s) \sqrt{V_s}  \left(\sqrt{1-\rho^2}dW^{(1)}_s+\rho dW^{(2)}_s\right)ds, \quad t \in [0,T] .
	\end{equation}
	Then $\left(\exp\left( U_t - \frac12 \langle U \rangle_t \right)\right)_{t\in[0,T]}$ is a martingale.
\end{lemma}

\begin{proof}
	As in the proof of \cite[Lemma 7.3]{Jaber2019affine} we define  
	for all $\widetilde T \in [0,T]$ and all $n \in \N$ the stopping times $\tau_n = \inf\{t\geq0\colon V_t >n\} \wedge \widetilde T$. 
	Consider $\check b \colon [0,T]\times [0,\infty) \times \Omega \to \R,$ $\check b(t,x,\omega) =  \kappa(t)\theta(t) - (\kappa(t)-\bar\sigma(t)g(t)1_{\{t\leq \tau_n(\omega)\}})x$ and $\check \sigma \colon [0,T]\times [0,\infty) \to \R,$ $\check \sigma(t,x) = \bar\sigma(t)\sqrt{x}$.  
	Since $\kappa$, $\theta$ and $\bar\sigma$ are continuous on $[0,T]$ and hence bounded on $[0,T]$, there exists $c_{LG} \in (0,\infty)$ such that for all $t \in [0,T]$, $x \in [0,\infty)$, $\omega \in \Omega$, $n \in \N$ it holds   $\lvert \check b(t,x,\omega) \rvert + \lvert\check \sigma(t,x) \rvert \leq c_{LG} (1+\lvert x\rvert)$. 
	Therefore, it follows from \cite[Remark 3.2 and Lemma 3.1]{Jaber2019affine} that the conditions of \cite[Lemma 2.4]{Jaber2019affine} are satisfied. 
	We can hence follow the proof of \cite[Lemma 7.3]{Jaber2019affine} to obtain the same claim also in the inhomogeneous Volterra-Heston model.
\end{proof}

\bibliographystyle{abbrv}
\bibliography{literature}

\end{document}